\documentclass[12pt]{amsart}

\usepackage{amsmath}
\usepackage{amssymb}
\usepackage{amsthm}
\usepackage{amscd}
\usepackage{xypic}
\swapnumbers \textwidth=16.00cm \textheight=22cm \topmargin=0.00cm
\headsep=1cm \numberwithin{equation}{section}
\newtheorem{theorem}{Theorem}[section]
\newtheorem{lemma}[theorem]{Lemma}
\newtheorem{proposition}[theorem]{Proposition}
\newtheorem{corollary}[theorem]{Corollary}

\theoremstyle{definition}
\newtheorem{definition}[theorem]{Definition}
\newtheorem{definition and remark}[theorem]{Definition and Remark}
\newtheorem{remark}[theorem]{Remark}
\newtheorem{remark and definition}[theorem]{Remark and Definition}
\newtheorem{remark and notation}[theorem]{Remark and Notation}
\newtheorem{notation and remark}[theorem]{Notation and Remark}
\newtheorem{notation and convention}[theorem]{Notation and Convention}

\newtheorem{reminder}[theorem]{Reminder}

\newtheorem{notation and remarks}[theorem]{Notation and Remarks}
\newtheorem{notation and reminder}[theorem]{Notation and Reminder}
\newtheorem{example}[theorem]{Example}

\newtheorem{construction and examples}[theorem]{Construction and Examples}

\newtheorem{problem and remark}[theorem]{Problem and Remark}
\newtheorem{convention and notation}[theorem]{Convention and Notation}
\newtheorem{convention and remark}[theorem]{Convention and Remark}

\newcommand\Proj{\operatorname{Proj}}

\newcommand\Hom{\operatorname{Hom}}

\newcommand\Tor{\operatorname{Tor}}

\newcommand\depth{\operatorname{depth}}

\newcommand\reg{\operatorname{reg}}
\newcommand\Ker{\operatorname{Ker}}
\newcommand\Coker{\operatorname{Coker}}
\newcommand\e{\operatorname{e}}

\newcommand\Nor{\operatorname{Nor}}

\newcommand\Sing{\operatorname{Sing}}

\newcommand\sreg{\operatorname{sreg}}

\newcommand\sat{\operatorname{sat}}
\newcommand\Reg{\operatorname{Reg}}
\newcommand\CM{\operatorname{CM}}
\newcommand\length{\operatorname{length}}

\begin{document}

\title[ON SURFACES OF MAXIMAL SECTIONAL REGULARITY]
         {ON SURFACES OF MAXIMAL SECTIONAL REGULARITY}

\author{Markus BRODMANN, Wanseok LEE, Euisung PARK, Peter SCHENZEL}

\address{Universit\"at Z\"urich, Institut f\"ur Mathematik, Winterthurerstrasse 190, CH -- Z\"urich, Switzerland}
\email{brodmann@math.unizh.ch}

\address{Pukyong National University, Department of applied Mathematics, Daeyeon Campus 45, Yongso-ro, Nam-Gu, Busan, Republic of Korea}
\email{wslee@pknu.ac.kr}

\address{Korea University, Department of Mathematics, Anam-dong, Seongbuk-gu, Seoul 136-701, Republic of Korea}
\email{euisungpark@korea.ac.kr}

\address{Martin-Luther-Universit\"at Halle-Wittenberg,
Institut f\"ur Informatik, Von-Secken\-dorff-Platz 1, D -- 06120
Halle (Saale), Germany} \email{schenzel@informatik.uni-halle.de}

\date{Busan, Halle, Seoul and Z\"urich, \today}



\begin{abstract} We study projective surfaces $X \subset \mathbb{P}^r$ (with $r \geq 5$) of maximal sectional regularity and degree $d > r$, hence surfaces for which the
Castelnuovo-Mumford regularity $\reg(\mathcal{C})$ of a general hyperplane section curve $\mathcal{C} = X \cap \mathbb{P}^{r-1}$ takes the maximally possible value $d-r+3$. We
use the classification of varieties of maximal sectional regularity of \cite{BLPS1} to see that these surfaces are either particular divisors on a smooth rational $3$-fold scroll $S(1,1,1)\subset \mathbb{P}^5$, or else admit a plane $\mathbb{F} = \mathbb{P}^2 \subset \mathbb{P}^r$ such that $X \cap \mathbb{F} \subset \mathbb{F}$ is a pure curve of degree $d-r+3$. We show that our surfaces are either cones over curves of maximal regularity, or almost non-singular projections of smooth rational surface scrolls. We use this to show that the Castelnuovo-Mumford regularity of such a surface $X$ satisfies the equality $\reg(X) = d-r+3$ and we compute or estimate various of the cohomological invariants as well as the Betti numbers of such surfaces. We also study the geometry of extremal secant lines of our surfaces $X$, more precisely the closure $\Sigma(X)$ of the set of all proper extremal secant lines to $X$ in the Grassmannian $\mathbb{G}(1, \mathbb{P}^r).$
\end{abstract}

\maketitle \thispagestyle{empty}

\section{Introduction}

\subsection*{Varieties of maximal sectional regularity}  In \cite{BLPS1} we have studied and classified projective varieties $X \subset \mathbb{P}^r$ of dimension $n \geq 2$, of codimension $c \geq 3$ and of degree $d \geq c+3$ which are of maximal sectional regularity, which means that the Castelnuovo-Mumford regularity $\reg(\mathcal{C})$ of a general linear curve section $\mathcal{C} = X \cap \mathbb{P}^{c+1} \subset \mathbb{P}^r$ of $X$ takes the maximally possible value $d-c+1$. There are two possible types of such varieties, namely (see also Theorem~\ref{theorem 2.1} below):\\
Either it holds $c=3$ and $X \approx H + (d-3)F$ is a divisor on a rational $(n+1)$-fold scroll $W \subset \mathbb{P}^{n+3}$ with $n-3$-dimensional vertex, where $H \subset W$ is the hyperplane divisor and $F = \mathbb{P}^n \subset W$ is a linear $n$-space; \\
or else, there is linear subspace $\mathbb{F} = \mathbb{F}(X) = \mathbb{P}^n \subset \mathbb{P}^r$ such that $X \cap \mathbb{F} \subset \mathbb{F}$ is a hypersurface of degree $d-c+1$.\\
If $X$ is of type II, the $n$-space $\mathbb{F}(X)$ is unique and coincides with the so-called extremal variety of $X$, that is the closed union of all lines in $\mathbb{P}^r$ which are $(d-c+1)$-secant to a general curve section of $X$.
Moreover, if the (algebraically closed) base field $\Bbbk$ is of characteristic $0$ or if $n = 2$, each variety $X \subset \mathbb{P}^r$ of maximal sectional regularity is sectionally smooth rational (in the sense of Section 2 below)  and hence is an almost non-singular linear projection of a rational $n$-fold scroll $\widetilde{X} \subset  \mathbb{P}^{d+n-1}$ (see Theorem~\ref{theorem 2.5}). In addition, this projecting scroll $\widetilde{X}$ is singular if and only if $X$ is a cone.

\subsection*{Surfaces of maximal sectional regularity} In this paper we focus on the case in which $n=2$, hence the case in which $X$ is a surface of maximal sectional regularity, and we shall investigate in detail the structure of $X$. In this case, the above two possible types present themselves as follows (see Corollary~\ref{corollary 2.1'} below):\\
 \textit{Type I:} It holds $r = 5$ and $X \approx H + (d-3)F$ is a smooth divisor on the smooth rational $3$-fold scroll $W = S(1,1,1) \subset \mathbb{P}^5$, where $H \subset W$ is the hyperplane divisor and $F = \mathbb{P}^2 \subset W$ is a ruling plane. \\
\textit{Type II:} There is a plane $\mathbb{F} = \mathbb{P}^2 \subset \mathbb{P}^r$ such that $X \cap \mathbb{F} \subset \mathbb{F}$ is a pure curve of degree $d-r+3$. \\ Moreover in this situation either the projecting surface scroll $\widetilde{X} \subset \mathbb{P}^{d+1}$ is smooth or $X$ is a cone over a curve of maximal regularity. It turns out that the surfaces in question have a rich geometric, homological and cohomological structure, which we aim to investigate in this paper.

\subsection*{Preview of results} \textit{Section 2:} We give a few preliminaries, which mainly rely on results established in \cite{BLPS1}. We also establish a bound for the Castelnuovo-Mumford regularity of varieties which are almost non-singular projections of varieties satisfying the Green-Lazarsfeld property $N_{2,p}$ (see Theorem~\ref{2.5' Theorem}). We shall revisit the extremal secant locus $\Sigma(X)$ of an arbitrary non-degenerate irreducible projective variety $X \subset \mathbb{P}^r$ of degree $d$ and codimension $c \geq 2$, that is the closure of the set of all proper $(d-c+1)$-secant lines to $X$ in the Grassmannian $\mathbb{G}(1,\mathbb{P}^r)$ of all lines in $\mathbb{P}^r$. This locus is particularly interesting if $X$ is a surface of extremal regularity and hence satisfies the inequality $\mathrm{reg}(X) \geq d-r+3$. We also consider the so-called special extremal locus ${}^*\Sigma(X)$ of a variety $X$ of maximal sectional regularity, hence the closure of the set of all lines in $\mathbb{G}(1,\mathbb{P}^r)$ which are $(d-r+3)$-secant to a general curve section of $X$. We show that this latter locus has dimension $2$, if $X$ is a surface of maximal sectional regularity (see Proposition~\ref{prop:dimspecextrloc}). \\
\textit{Section 3:} We study sectionally smooth rational surfaces, hence surfaces whose general curve section is smooth and rational -- a property which holds for surfaces of maximal sectional regularity.  As sectionally smooth rational surfaces are almost non-singular projections of surface scrolls, they have a number of interesting properties and their cohomology is quite  well understood. In particular, they satisfy the conjectural regularity bound of Eisenbud-G\^oto \cite{EG} (see Theorem~\ref{theorem 2.9}). The results of this section will pave our way for a more detailed investigation of surfaces of maximal sectional regularity.\\
\textit{Section 4}: We investigate surfaces of maximal sectional regularity of type I. In particular, we compute their Betti tables (see Theorem~\ref{t1-betti}) and their cohomological Hilbert functions (see Theorem~\ref{t1-coh}). Moreover we show that the special extremal secant locus of such surfaces coincides with their extremal locus, and we show that these loci become Veronese Surfaces in a projective $5$-space under the Pl\"ucker embedding (see Proposition~\ref{prop:extseclocI}). \\
\textit{Section 5:} We study surfaces of maximal sectional regularity which fall under type II. We notably investigate the cohomological invariants and the cohomology tables of these surfaces (see Theorem~\ref{4.14'' Theorem} and Corollary~\ref{coro:cohomology}). If $X$ is a variety of maximal sectional regularity of dimension $n \geq 2$, which falls under type II, the union $Y := X \cup \mathbb{F}(X)$ of $X$ with its extremal variety $\mathbb{F}(X) = \mathbb{P}^n \subset \mathbb{P}^r$ plays a crucial role. As an application of Theorem~\ref{4.14'' Theorem}, we shall establish a lower bound on the number of defining quadrics of a variety of maximal sectional regularity $X$ of type II with arbitrary dimension $n \geq 2$ -- a bound which is sharp if and only $Y$ is arithmetically Cohen-Macaulay (see Corollary~\ref{cor:higherdim}). Finally, in the surface case, we give a comparison result for the Betti numbers of $X$ and $Y$ (see Proposition~\ref{prop:BettiNumbers}). \\
\textit{Section 6:} We study the index of normality $N(X)$ of a surface $X$ of maximal sectional regularity which falls under type II. In all examples we found, this index is sub-maximal and hence satisfies the inequality $N(X) \leq d-r$. In Theorem~\ref{4.17'' Proposition} we give various conditions which are equivalent to the mentioned sub-maximality of $N(X)$. This sub-maximality notably implies that the homogeneous vanishing ideal of the union $Y = X \cup \mathbb{F}(X)$ is generated in degrees ~$\leq d-r+2$.~This latter property allows to draw conclusions on the geometry of extremal secant lines to $X$ (see Remark~\ref{remark:normality}). We also revisit surfaces of degree $r+1$ in $\mathbb{P}^r$ and prove, that two of the eleven cases listed in \cite{B2} and \cite{BS6} may indeed not occur, as conjectured (see Remark~\ref{3.3' Remark}). \\
\textit{Section 7:} This section is devoted to examples and open
problems. We first provide examples of large families of surfaces of
extremal regularity which are not of maximal sectional regularity
and whose extremal secant locus is of any dimension in the maximally
possible range $\{-1,0,1\}.$ (see Construction and Examples~\ref{7.1
Construction and Examples}). This is of some interest, as the paper
\cite{GruLPe} let to the expectation that ``there are only a few
``exceptional" varieties of extremal regularity without extremal
secant lines". We also provide some examples which make explicit the
Betti tables in the case of surfaces of maximal sectional regularity
which fall under type I (see Example~\ref{example:t1}). We also
suggest a general construction principle which provides large
classes of surfaces of maximal sectional regularity of type II (see
Construction and Examples~\ref{7.2 Construction and Examples}). We
use this principle to produce explicit examples for which we compute
the Betti tables (see Examples~\ref{7.3 Example}, ~\ref{7.4 Example}
and ~\ref{7.5 Example}). Finally, we give some conclusive remarks
and suggest a few open problems, which are related to the previously
mentioned question on the sub-maximality of the index of normality
(see Problems and Remark~\ref{7.6 Problem and Remark}).

\section{Preliminaries}

\subsection*{The classification of varieties of maximal sectional regularity}

Let $X \subset \mathbb{P}^r$ be a non-degenerate irreducible projective variety of dimension $n \geq 2$, codimension $c \geq 3$ and degree $d \geq c+3$. We recall the following classification result on varieties of maximal sectional regularity, which was established in \cite[Theorem 7.1]{BLPS1}

\begin{theorem}
\label{theorem 2.1}
If either $n = 2$ or $\mathrm{Char}(\Bbbk) = 0$, the variety $X \subset \mathbb{P}^r$ is of maximal sectional regularity if and only if it falls under one of the following two types:
\begin{itemize}
\item[{}]{\bf \textit{Type I:}} $c=3$ and $X$ is a divisor of the $(n+1)$-fold scroll
$$W := S(\underbrace{0,\ldots,0}_{(n-2)-\rm{times}},1,1,1) \subset \mathbb{P}^{n+3}$$
with $X \approx H + (d-3)F$, where $H$ is the hyperplane divisor of $W$ and $F \subset W$ is a linear subspace of dimension $n$.
\item[{}]{\bf \textit{Type II:}} There exists an $n$-dimensional linear subspace $\mathbb{F} = \mathbb{P}^n \subset \mathbb{P}^r$ such that $X \cap \mathbb{F} \subset \mathbb{F}$ is a hypersurface of degree $d-c+1$.
\end{itemize}
\end{theorem}

\begin{remark}
\label{remark 2.2}
The previous classification result allows to conclude that
there exist varieties $X \subset \mathbb{P}^r$ of maximal sectional regularity of dimension $n$, of codimension $c$
and of degree $d$ for any given triplet $(n,c,d)$ with $n \geq 2$, $c \geq
3$ and $d \geq c+3$. \\
\end{remark}

For the purposes of the present paper, we notice in particular the following result (see also \cite[Theorem 6.3]{BLPS1}):

\begin{corollary}
\label{corollary 2.1'} Let $5 \leq r < d$. Then, a non-degenerate irreducible projective surface $X \subset \mathbb{P}^r$ of degree $d$ is of maximal sectional regularity if and only if it falls under one of the following two types:
\begin{itemize}
\item[{}] {\bf \textit{Type I:}} $r=5$ and $X$ is a divisor of the smooth $3$-fold scroll $W := S(1,1,1) \subset \mathbb{P}^5$ with $X \approx H + (d-3)F$, where $H$
is the hyperplane divisor of $W$ and $F \subset W$ is a ruling plane. In this case, the surface $X$ is smooth.
\item[{}] {\bf \textit{Type II:}} There exists a plane $\mathbb{F} = \mathbb{P}^2 \subset \mathbb{P}^5$ such that $X \cap \mathbb{F} \subset \mathbb{F}$ is a pure curve of degree $d-r+3$. In this case, the surface $X$ is singular.
\end{itemize}
\end{corollary}

We now introduce the notion of sectional regularity and characterize surfaces of maximal sectional regularity by means of this invariant.

\begin{remark and definition}
\label{remark and definition 0} (A) Let $X \subset \mathbb{P}^r = \Proj(S := \Bbbk[x_0,x_1,\ldots,x_r])$ be a non-degenerate irreducible projective variety of dimension $n \geq 2$, of codimension $c \geq 2$ and of degree $d$. We introduce the notation
$$\mathbb{H}_h := \Proj(S/hS) \mbox{ for all } h \in S_1 = \sum_{i=0}^r \Bbbk x_i \mbox{ with } h \neq 0\},$$
and we define the \textit{sectional regularity} of $X$ by
$$\mathrm{sreg}(X) := \mathrm{min}\{\mathrm{reg}(X \cap \mathbb{H}_h) \mid h \in S_1 \setminus \{0\}\}.$$
As the regularity is semi-continuous on hyperplane sections, we can say that
$$\mathbb{W}(X) := \{h \in S_1 \setminus \{0\} \mid \mathrm{reg}(X \cap \mathbb{H}_h) = \mathrm{sreg}(X)\} \mbox{ is a dense open subset of } S_1 \setminus\{0\}.$$
Now, there is a dense open subset $\mathbb{U} \subseteq \mathbb{W}$ such that $X \cap \mathbb{H}_h$ is an integral scheme, and we denote the largest of these open sets by $\mathbb{U}(X).$\\
(B) Now, assume that $4 \leq r < d$ and that $X \subset \mathbb{P}^r = \Proj(S := \Bbbk[x_0,x_1,\ldots,x_r])$ is a non-degenerate irreducible projective surface of degree $d$. In this situation
$$\mathcal{C}_h := X \cap \mathbb{H}_h \mbox{ is an integral curve of degree $d$ with } \mathrm{reg}(\mathcal{C}_h) = \mathrm{sreg}(X) \mbox{ for all } h \in \mathbb{U}(X).$$
Hence, in particular we see that
$$\mathrm{sreg}(X) \leq d-r+3 \mbox{ with equality if and only if } X \mbox{ is of maximal sectional regularity}.$$
\end{remark and definition}

\subsection*{Curves of maximal regularity} As the generic linear curve sections of varieties of maximal sectional regularity are curves of maximal regularity, it will be useful for us to keep in mind the following fact.

\begin{proposition}
\label{proposition 2.3} Let $r\geq 4$ and let $\mathcal{C} \subset
\mathbb{P}^r$ be a curve of degree $d \geq r+2$ which is of maximal
sectional regularity, so that $\reg(\mathcal{C}) = d-r+2$. Then
$\mathcal{C}$ admits a unique $(d-r+2)$-secant line $\mathbb{L}$.
Moreover, if $d \geq 3r-3,$ then  $\mathrm{depth}(\mathcal{C} \cup
\mathbb{L}) = 1.$
\end{proposition}
\begin{proof}
The existence of the $(d-r+2)$-secant line $\mathbb{L}$ follows from the classification of curves of maximal regularity given in \cite{GruLPe}. The uniqueness of the extremal secant line $\mathbb{L}$ follows by \cite[(3.1)]{BS2}. \\
Assume now that $d \geq 3r-3$. Note that $\mathbb{S} := \mathrm{Join}(\mathbb{L},C) \subset \mathbb{P}^r$ is a rational normal $3$-fold scroll of type $S(0,0,r-1)$ whose vertex $S(0,0) \subset S(0,0,1)$ equals $\mathbb{L}.$ So, in degree $2$, the homogeneous vanishing ideals $I_{\mathcal{C}}$ and $I_{\mathbb{S}}$ of $\mathcal{C}$ respectively of  $\mathbb{S}$ in $S := \Bbbk[x_0.x_1,\ldots,x_r]$ satisfy the relation
$$\mathrm{dim}_{\Bbbk}(I_{\mathcal{C}})_2 \geq \mathrm{dim}_{\Bbbk}(I_{\mathbb{S}})_2 = \binom{r-2}{2}.$$
Assume now that $\mathrm{depth}(\mathcal{C} \cup \mathbb{L}) \neq 1$, so that
$\mathcal{C} \cup \mathbb{L}$ is arithmetically Cohen-Macaulay. Then, by \cite[Proposition 3.3]{BS2} it follows that
$$\mathrm{dim}_\Bbbk(I_C)_2 = \binom{r+1}{2}-d-1, \mbox{ whence } \binom{r-2}{2} \leq \binom{r+1}{2}-d-1,$$
and this yields the contradiction that $d \leq 3r-4.$
\end{proof}

\subsection*{Sectionally rational varieties} It is most important, that each variety $X$ of maximal sectional regularity is  \textit{sectionally rational}, which means that its general curve section is rational. If the general curve section of $X$ is even smooth and rational, we say that $X$ is \textit{sectionally smooth rational}. A particularly interesting property of sectionally rational varieties is the fact, that they are birational linear projections of varieties of minimal degree. To make this
statement more precise, we first give the following definition.

\begin{definition and remark}
\label{definition and remark 2.4} (A) We define the \textit{singular locus} of a finite morphism $f: X' \longrightarrow X$ of noetherian schemes by
$$\mathrm{Sing}(f) := \{x \in X \mid \mathrm{length}\big(f^{-1}(x)\big) \geq 2\}.$$
Observe, that we also may write
$$\mathrm{Sing}(f) = \mathrm{Supp}\big((f_*\mathcal{O}_{X'})/\mathcal{O}_X\big).$$

(B) We say that the finite morphism $f: X' \longrightarrow X$ is \textit{almost non-singular} if its singular locus $\mathrm{Sing}(f)$ is a finite set.
\end{definition and remark}

Now, we have the following result (see \cite[Theorem 4.1]{BLPS1}).

\begin{theorem}
\label{theorem 2.5} Let $X \subset \mathbb{P}^r$ be a non-degenerate irreducible sectionally rational projective variety of dimension $n\geq 2$ and degree $d$. Assume furthermore that $\rm{char}(\Bbbk)=0$ or $n=2$. Then, we may write
\begin{itemize}
\item[(a)] $X = \pi_{\Lambda}(\widetilde{X})$, where $\widetilde{X} \subset \mathbb{P}^{d+n-1}$
is an $n$-dimensional variety of minimal degree,
\item[(b)] $\Lambda = \mathbb{P}^{d+n-r-2} \subset \mathbb{P}^{d+n-1}$ is a linear subspace with $\widetilde{X} \cap \Lambda = \emptyset,$
\item[(c)] $\pi_{\Lambda}:\mathbb{P}^{d+n-1} \setminus \Lambda \longrightarrow \mathbb{P}^r$ is the linear projection map from $\Lambda$ and
\item[(d)] the induced finite morphism $\pi_{\Lambda}:\widetilde{X} \rightarrow X$ is the normalization of $X$.
\end{itemize}
Moreover, if $d\geq 5$, then $\widetilde{X}$ is a rational $n$-fold scroll. Finally, if $X$ is a sectionally smooth rational surface, the morphism $\pi_{\Lambda}:\widetilde{X} \rightarrow X$ is almost non-singular.
\end{theorem}

\subsection*{A regularity bound for almost non-singular projections} We know by Theorem~\ref{theorem 2.5}, that sectionally smooth rational surfaces are almost non-singular linear projections of rational normal scrolls. This will allow us to prove that these surfaces satisfy the conjectural Eisenbud-Goto bound. In this subsection, we shall actually prove a much more general bounding result for the regularity of almost non-singular linear projections.

\begin{definition}
\label{definition 2.6}
(A) Let $p \in \mathbb{N}$. The graded ideal $I \subset S := \Bbbk[x_0,x_1,\ldots,x_r]$ is said to satisfy
the \textit{(Green-Lazarsfeld) property} $N_{2,p}$ (see \cite{GL}) if the Betti numbers of $S/I$ satisfy the condition
$$\beta_{i,j} := \beta^S_{i,j}(S/I) = \beta^S_{i-1,j+1}(I) = 0 \mbox{ whenever } i \leq p \mbox{ and } j \neq 1,$$
-- hence if the minimal free resolution of $I$ is linear up to the homological degree $p$ -- and thus has the form
$$\ldots \rightarrow S^{\beta_{p,1}}(-p-1)\rightarrow \ldots \rightarrow S^{\beta_{1,1}}(-2) \rightarrow I \rightarrow 0.$$
(B) The closed subscheme $Z \subset \mathbb{P}^r$ is said to satisfy the \textit{property} $N_{2,p}$ if its homogeneous vanishing ideal $I_Z \subset S$ satisfies the property $N_{2,p}$.
\end{definition}
Now, we may prove the announced regularity bound for almost non-singular projections of $N_{2,p}$-varieties.

\begin{theorem}\label{2.5' Theorem} Let $r' \geq r$ be integers, let $X' \subset \mathbb{P}^{r'}$ be a non-degenerate
projective variety of dimension $ \geq 2$ which satisfies the property $N_{2,p}$ for some $p \geq\max\{2,r'-r+1\}$. Let $\Lambda = \mathbb{P}^{r'-r-1}$
be a subspace such that $X' \cap \Lambda = \emptyset$ and let $\pi_\Lambda:\mathbb{P}^{r'} \setminus \Lambda: \twoheadrightarrow \mathbb{P}^r$ be the linear projection from $\Lambda$. Let $X := \pi_\Lambda(X') \subset \mathbb{P}^r$ and assume that the induced finite morphism
$\pi_{\Lambda}: X' \twoheadrightarrow X$ is almost non-singular. Then
\begin{itemize}
\item[\rm{(a)}] The homogeneous vanishing ideal $I_X \subset S$ of $X$ is generated by homogeneous polynomials of degrees $\leq r'-r+2$.
\item[\rm{(b)}] $\reg(X) \leq \max\{\reg(X'), r'-r+2\}$.
\end{itemize}
\end{theorem}

\begin{proof} Let $I_{X'} \subset S' := \Bbbk[x_0,x_1,\ldots,x_r,x_{r+1},\ldots,x_{r'}] = S[x_{r+1},\ldots,x_{r'}]$ be the homogeneous vanishing ideal of $X' \subset \mathbb{P}^{r'} = \Proj(S')$ and let $A':= S'/I_{X'}$ be the homogeneous coordinate ring of $X'$. We assume that $\Lambda = \Proj(S'/(x_0,x_1,\ldots,x_r)S')$, consider $A'$ as a finitely generated graded $S$-module and set $t:= r'-r$. As $X'$ satisfies the condition $N_{2,p}$ with $p \geq\max\{2,t+1\}$,
it follows by \cite[Theorem 3.6]{AK}, that the minimal free presentation of $A'$ has the shape
$$S^s(-2) \stackrel{v}{\rightarrow} S \oplus S^t(-1) \stackrel{q}{\rightarrow} A' \rightarrow 0$$
for some $s \in \mathbb{N}$. Moreover, the coordinate ring $A = S/I_X$ of $X$ is nothing else than the image $q(S)$ under $q$ of the direct summand
$S \subset S\oplus S^t(-1)$. Therefore
$$A'/A \cong \Coker\big(u:S^s(-2) \rightarrow S^t(-1)\big),$$
where $u$ is the composition of the map $v:S^s(-2) \rightarrow S\oplus S^t(-1)$ with the canonical projection map $w: S\oplus S^t(-1) \twoheadrightarrow S^t(-1)$.
Hence, the $S$-module $(A'/A)(1)$ is generated by $t$ homogeneous elements of degree $0$ and related in degree $1$. As $\Sing(\pi_{\Lambda})$ is finite,
we have $\dim(A'/A) \leq 1$. So, it follows by \cite[Corollary 2.4]{ChFN} that $\reg\big((A'/A)(1)\big) \leq t-1$, whence $\reg(A'/A) \leq t$. Now, the short
exact sequence $0\rightarrow A \rightarrow A' \rightarrow A'/A \rightarrow 0$ implies that $\reg(A) \leq \max\{\reg(A'), t+1\}$.  It follows that
$\reg(X) = \reg(A) + 1 \leq \max\{\reg(A') + 1, t+2\} = \max\{\reg(X'), t+2\} = \max\{\reg(X'), r'-r+2\}$. This proves claim (b).\\
To prove claim (a), observe that $I_X = \Ker(q) \cap S$ occurs in the short exact sequence of graded $S$-modules
$0 \rightarrow I_X \rightarrow \mathrm{Im}(v) \stackrel{w\upharpoonright}{\rightarrow} \mathrm{Im}(u) \rightarrow 0$, where ${w\upharpoonright}$ is the
restriction of the above projection map $w$. In particular, we may identify ${w\upharpoonright}$ with the canonical map $S^s(-2)/\Ker(v) \twoheadrightarrow S^s(-2)/\Ker(u)$.
It follows, that $I_X \cong \Ker(u)/\Ker(v)$. In view of the exact sequence
$0\rightarrow \Ker(u) \rightarrow S^s(-2) \stackrel{u}{\rightarrow}S^t(-1) \rightarrow A'/A \rightarrow 0$, we finally get $\reg(\Ker(u)) \leq t+2 = r'-r+2$.
Therefore $\Ker(u)$ is generated in degrees $\leq r'-r+2$, and hence so is $I_X$. This proves statement (a).
\end{proof}

\subsection*{Extremal secant loci and extremal varieties} In this subsection, we recall a few facts on the geometry of proper $(d-c+1)$-secant lines to a non-degenerate irreducible projective variety $X \subset \mathbb{P}^r$ of codimension $c$ and degree $d$. We also recall the related notion of extremal secant locus $\Sigma (X)$ of $X$, that is, the closure of the set of all proper $(d-c+1)$-secant lines of $X$ in the Grassmannian $\mathbb{G}(1,\mathbb{P}^r)$.

\begin{notation and reminder}\label{4.1'' Notation and Reminder} (See \cite[Notation and Reminder 3.1]{BLPS1}) (A) Let $X \subset \mathbb{P}^r$ be as above and let
$$\Sigma(X) := \overline{\{\mathbb{L} \in \mathbb{G}(1,\mathbb{P}^r) \mid d-c+1 \leq \mathrm{length}(X \cap \mathbb{L}) < \infty\}}$$
denotes the \textit{extremal secant locus} of $X.$
Keep in mind, that setting $n := \dim(X) = r-c$ we can say (see \cite[Theorem 3.4]{BLPS1})
$$\dim\big(\Sigma(X)\big) \leq 2n-2 \mbox{ with equality if and only if } X \mbox{ is of maximal sectional regularity}.$$
(B) Keep the above notations and hypotheses and let $\mathcal{U}(X)$ denote the largest open subset $\mathcal{U} \subset \mathbb{G}(c+1, \mathbb{P}^r )$ such that
$$\mathcal{C}_{\Lambda} := X \cap \Lambda \subset \Lambda = \mathbb{P}^{c+1} \mbox{ is an integral curve of maximal regularity for all } \Lambda \in \mathcal{U}.$$
Observe that
$$X \mbox{ is of maximal sectional regularity if and only if } \mathcal{U}(X) \neq \emptyset.$$
\end{notation and reminder}

We introduce a subset of the extremal locus of a variety of maximal sectional regularity, which reflects in a particular way the nature of these varieties. We use this set to define the extremal variety of a variety of maximal sectional regularity.

\begin{notation and reminder}\label{4.2'' Notation and Reminder} (See \cite[Section 5]{BLPS1})
(A) Let the notations and hypotheses be as in Notation and Reminder~\ref{4.1'' Notation and Reminder}, and assume that $X$ is of dimension $n \geq 2$ and of maximal sectional regularity, so that $\mathcal{U}(X) \neq \emptyset$. For all $\Lambda \in \mathcal{U}(X)$ let $\mathbb{L}_{\Lambda} \in \Sigma(X)$ denote the unique $(d-c+1)$-secant line to the curve $\mathcal{C}_{\Lambda} \subset \Lambda = \mathbb{P}^{c+1}$ (see Notation and Reminder~\ref{4.1'' Notation and Reminder} (B) and Proposition~\ref{proposition 2.3}), so that
$$\mathbb{L}_{\Lambda} \in \mathbb{G}(1, \mathbb{P}^r) \mbox{ with } \quad \mathrm{length}(\mathcal{C}_{\Lambda} \cap \mathbb{L}_{\Lambda}) = \mathrm{length}(X \cap \mathbb{L}_{\Lambda}) = d-c+1.$$
The $(d-c+1)$-secant lines of the form $\mathbb{L}_{\Lambda}$ with $\Lambda \in \mathcal{U}(X)$ are called \textit{special extremal secant lines}, whereas the set
$${}^*\Sigma(X) := \overline{\{\mathbb{L}_{\Lambda} \mid \Lambda \in \mathcal{U}(X)\}} \subseteq \Sigma(X) \subset \mathbb{G}(1,\mathbb{P}^r)$$
is called the \textit{special extremal secant locus} of $X$. If $n=1$, then $4 \leq r < d$ and $X \subset \mathbb{P}^r$ is a curve of maximal sectional regularity, and hence admits a unique extremal secant line $\mathbb{L} \in \mathbb{G}(1,\mathbb{P}^r)$ (see Proposition~\ref{proposition 2.3}). So, we define ${}^*\Sigma(X) := \{\mathbb{L}\}$ in this case. \\
(B) We define the \textit{extremal variety} and the \textit{extended extremal variety} of $X$ respectively by
$$\mathbb{F}(X) := \overline{\bigcup_{\Lambda \in \mathcal{U}(X)} \mathbb{L}_{\Lambda}} \quad (\quad \subseteq \quad) \quad \mathbb{F}^{+}(X) := \overline{\bigcup_{\mathbb{L} \in \Sigma(X)} \mathbb{L}}.$$
(C) Keep the previous notations and hypotheses, and assume that $5 \leq r < d$. Let $X \subset \mathbb{P}^r = \Proj\big(S := \Bbbk[x_0,x_1,\ldots,x_r]\big)$ be a surface of maximal sectional regularity, so that $c+1 = r-1$ and $\mathrm{sreg}(X) = d-r+3$. Then, in the notations of Remark and definition~\ref{remark and definition 0} we have
$$\mathbb{U}(X) = \{h \in S_1 \setminus\{0\} \mid \Proj(S/hS) \in \mathcal{U}(X)\}.$$
Moreover, for each $h \in \mathbb{U}(X)$, the line $\mathbb{L}_h := \mathbb{L}_{\mathbb{H}_h}$ is the unique $(d-r+3)$-secant line to the curve of maximal regularity $\mathcal{C}_h \subset \mathbb{H}_h$, and hence the line defined by the condition
$$\mathbb{L}_h \in \mathbb{G}(1, \mathbb{P}^r)\quad \mbox{ with } \quad \mathrm{length}(\mathcal{C}_h \cap \mathbb{L}_h) = \mathrm{length}(X \cap \mathbb{L}_h) = d-r+3.$$
Now, according to \cite[Theorem 6.3]{BLPS1} we can say:
\begin{itemize}
\item[\rm{(a)}] If $X$ is of type I, then the extremal variety $\mathbb{F}(X)$ and the extended extremal variety $\mathbb{F}^{+}(X)$ of $X$ both coincide with the smooth $3$-fold scroll $W = S(1,1,1) \subset \mathbb{P}^r$ of Corollary~\ref{corollary 2.1'}.
\item[\rm{(b)}] If $X$ is of type II, then the extremal variety $\mathbb{F}(X)$ of $X$ coincides with the plane $\mathbb{F} = \mathbb{P}^2 \subset \mathbb{P}^r$ of Corollary~\ref{corollary 2.1'}.
\end{itemize}
\end{notation and reminder}

The following result says that the extremal secant locus and the special extremal secant locus of a variety of maximal sectional regularity have the same dimension.

\begin{proposition}
\label{prop:dimspecextrloc} Let $c \geq 3$, let $d \geq c+3$, let $n \geq 1$  and $X \subset \mathbb{P}^r$ be a non-degenerate irreducible variety of dimension $n$ and degree $d$ which is of maximal sectional regularity. Then $\dim\big({}^*\Sigma(X)\big) = 2n-2$.
\end{proposition}
\begin{proof} As ${}^*\Sigma(X) \subseteq \Sigma(X)$ it follows by the last observation made in Notation and Reminder~\ref{4.1'' Notation and Reminder} (A) that
$\dim\big({}^*\Sigma(X)\big) \leq \dim\big(\Sigma(X)\big) \leq 2n-2.$ It remains to show, that $\dim\big({}^*\Sigma(X)\big) \geq 2n-2.$ \\
We proceed by induction on $n.$ If $n = 1$, then our claim is clear by the definition of ${}^*\Sigma(X)$ (see Notation and Reminder~\ref{4.2'' Notation and Reminder} (A)). So, let $n > 1$ and let $\mathbb{H} = \mathbb{P}^{r-1} \subset \mathbb{P}^r$ be a general hyperplane. Then $X \cap \mathbb{H} \subset \mathbb{H}$ is a variety of dimension $n - 1$, codimension $c$ and degree $d$, which is of maximal sectional regularity. Moreover, each special extremal secant line $\mathbb{L} \in {}^*\Sigma(X \cap \mathbb{H})$ to $X \cap \mathbb{H}$ is a special extremal secant line to $X$. Therefore, we can say that
${}^*\Sigma(X \cap \mathbb{H}) \subseteq {}^*\Sigma(X) \cap \mathbb{G}(1,\mathbb{H}).$\\
By induction, we have $\dim\big({}^*\Sigma(X \cap \mathbb{H})\big) \geq 2(n-1)-2 = 2n-4.$ If we apply \cite[Lemma 3.2]{BLPS1} with $T = {}^*\Sigma(X)$ we obtain
$\dim\big({}^*\Sigma(X)\big) \geq \dim\big({}^*\Sigma(X) \cap \mathbb{G}(1,\mathbb{H})) + 2.$ This proves our claim.
\end{proof}

\section{Sectionally Smooth Rational Surfaces}

\subsection*{The projecting scroll} In this section we investigate sectionally smooth rational surfaces. We do this, because surfaces of maximal sectional regularity are sectionally smooth rational. Let us recall first, that according to Theorem~\ref{theorem 2.5} we can say.

\begin{corollary}
\label{corollary 2.7}
 Let $X \subset \mathbb{P}^r$ be a non-degenerate irreducible projective and sectionally smooth rational surface of degree $d$ with $r \geq 4$ and $d \geq r+1$.\\ Then, there exists a unique non-negative integer $a \leq \frac{d}{2}$ such that $X = \pi_{\Lambda} (\widetilde{X})$, where $\widetilde{X} = S(a,d-a) \subset \mathbb{P}^{d+1}$ and
\begin{itemize}
\item[\rm{(a)}] $\Lambda = \mathbb{P}^{d-r} \subset \mathbb{P}^{d+1}$ is a subspace such that $\widetilde{X} \cap \Lambda = \emptyset$,
\item[\rm{(b)}] $\pi_{\Lambda} : \mathbb{P}^{d+1} \setminus \Lambda \rightarrow \mathbb{P}^r$ is the linear projection map from $\Lambda$,
\item[\rm{(c)}] the induced finite morphism $\pi_{\Lambda}: \widetilde{X} \rightarrow X$ is almost non-singular and coincides with the normalization of $X$, and
\item[\rm{(d)}] $\widetilde{X}$ is smooth (or, equivalently, $a > 0$) if and only if $X$ is not a cone.
\end{itemize}
\end{corollary}

\begin{definition}
In the above situation, we call $\widetilde{X}= S(a,d-a)$ the \textit{projecting scroll} of the surface $X$, $\Lambda  \subset \mathbb{P}^r$ the \textit{projecting center} for $X$ and $\pi_\Lambda:\widetilde{X} \twoheadrightarrow X$ the \textit{standard normalization} of $X$.\\
\end{definition}

\subsection*{Algebraic and cohomological properties} The precise aim of this section is to investigate a few algebraic and geometric properties of sectionally smooth rational surfaces, which are encoded in Corollary~\ref{corollary 2.7}. We begin with a few preliminaries.

\begin{notation and reminder}\label{2.7' Notation and Reminder}
Let $X \subset \mathbb{P}^r = \Proj(S := \Bbbk[x_0,\ldots,x_r])$ be a non-degenerate irreducible projective surface of degree $d$, homogeneous vanishing ideal $I \subset S$ and homogeneous coordinate ring $A = S/I$.

\noindent (A) (see \cite{BS6}) For any graded ideal $\mathfrak{b} \subseteq A_+ := S_+A$ and any graded $A$-module $M$ let $D_\mathfrak{b}(M) := \varinjlim \Hom_A({\mathfrak b}^n, M)$ denote the $\mathfrak{b}$-transform of $M$, and let $H^i_\mathfrak{b}(M), \quad (i \in \mathbb{N}_0)$ denote the $i$-th local cohomology module of $M$, both furnished with their natural grading. We usulally will write $H^i(M)$ instead of $H^i_{A_+}(M).$  \\
Let ${\mathfrak a} \subseteq A_+$ be the graded radical ideal which defines the non-Cohen-Macaulay locus $X \setminus \operatorname{CM}(X)$ of $X$. Observe that height ${\mathfrak a} \geq 2,$ so that the $\mathfrak{a}$-transform
\begin{equation*}
B(A):= D_{\mathfrak a}(A) = \varinjlim \Hom_A({\mathfrak a}^n, A) = \bigcup_{n\in \mathbb{N}}(A:_{\rm{Quot}(A)}\mathfrak{a}^n) =
\bigoplus _{n \in {\mathbb Z}} \Gamma (\operatorname{CM}(Z), {\mathcal O}_Z (n))
\end{equation*}
of $A$ is a positively graded finite birational integral extension domain of $A$. In particular $B(A)_0 = \Bbbk$. Moreover $B(A)$
has the second Serre-property $S_2$. As $\mbox{Proj}(B(A))$ is of dimension $2$, it thus is a locally Cohen-Macaulay scheme.

If $E$ is a finite graded integral extension domain of $A$ which satisfies the property $S_2$, we have $A \subset B(A) \subset E$. So $B(A)$ is the least finite
graded integral extension domain of $A$ which has the property $S_2$. Therefore, we call $B(A)$ the $S_2${\it -cover of} $A$. We also can describe $B(A)$ as the endomorphism
ring $\mbox{End}(K(A), K(A))$ of the canonical module $K(A) = K^3(A) = \mbox{Ext}^{r-2}_S(A, S(-r-1))$ of $A$.

\noindent (B) The inclusion map $A \rightarrow B(A)$ gives rise to a finite morphism
$$\pi: \widetilde{X}:= \Proj(B(A)) \twoheadrightarrow X, \mbox{ with } \Sing(\pi) = X \setminus \CM(X).$$
In particular $\pi$ is almost non-singular and hence birational. Moreover, for any finite morphism $\rho: Y \twoheadrightarrow X$ such that $Y$ is locally Cohen-Macaulay, there is a unique morphism $\sigma:Y \rightarrow \widetilde{X}$ such that $\rho = \pi \circ\sigma$. In addition $\sigma$ is an
isomorphism if and only if $\Sing(\rho) = X \setminus \CM(X)$. Therefore, the morphism $\pi:\widetilde{X} \twoheadrightarrow X$ is addressed as the \textit{finite Macaulayfication} of $X$. Keep in mind, that -- unlike to what happens with normalization -- there may be proper birational morphisms
$\tau:Z \twoheadrightarrow X$ with $Z$ locally Cohen-Macaulay, which do not factor through $\pi$ (see \cite{B0}).

\noindent (C) We also introduce the invariants
\begin{equation*}
\e_x(X) := \length(H^1_{\mathfrak{m}_{X,x}}(\mathcal{O}_{X,x})),
(x \in X \mbox{ closed }) \; \mbox{ and }
\e(X):=  \sum_{x\in X, {\rm{closed}}} \e_x(X).
\end{equation*}
Note that the latter counts the {\it number of non-Cohen-Macaulay points} of $X$
in a weighted way. Keep in mind that
$$\e(X) = h^1 (X, \mathcal{O}_X(j)) = h^2(\mathbb{P}^r,\mathcal{I}_X(j)) \mbox{ for all } j \ll 0.$$
\end{notation and reminder}

Now, we are ready to prove the following result on sectionally smooth rational surfaces.

\begin{theorem}
\label{theorem 2.9}
Let $X \subset \mathbb{P}^r$ be a non-degenerate irreducible projective and sectionally smooth rational surface of degree $d$ with $r \geq 4$ and $d \geq r+2$. Then it holds
\begin{itemize}
\item[\rm{(a)}] $\reg(X) \leq d - r + 3.$
\item[\rmfamily{(b)}] $\Reg(X) = \Nor(X) = \CM(X)$.
\item[\rm{(c)}] The homogeneous coordinate ring of the projecting scroll $\widetilde{X}\subset \mathbb{P}^{d+1}$ is the $S_2$-cover $B(A)$ of the homogeneous coordinate ring $A$ of $X \subset \mathbb{P}^r$ and the standard normalization $\pi_{\Lambda}: \widetilde{X} \rightarrow X$ is the finite Macaulayfication of $X$.

\item[\rm{(d)}] If $\e(X) = 0,$ then $h^2(\mathbb{P}^r,\mathcal{I}_X(j)) = 0$ for all $j \in \mathbb{Z}$. If $\e(X) \neq 0$, then
                $$h^2(\mathbb{P}^r,\mathcal{I}_X(j)) \begin{cases} = \e(X)                                                  &\mbox{ for all } j \leq 0;\\
                                                                  = \e(X) + h^1(\mathbb{P}^r,\mathcal{I}_X(1)) + r - d -1                 &\mbox{ for } j=1;\\
                                                                  \leq \max\{0, h^2(\mathbb{P}^r,\mathcal{I}_X(j-1)) -1 \}                 &\mbox{ for all } j > 1;\\
                                                                  \leq 1                                                                   &\mbox{ for } j = d-r;\\
                                                                  = 0                                                                      &\mbox{ for all } j \geq d-r+1.
                                                    \end{cases}$$
\item[\rm{(e)}] $h^3(\mathbb{P}^r, \mathcal{I}_X(j)) =
                h^2(\widetilde{X},\mathcal{O}_{\widetilde{X}}(j))$ for all $j \in \mathbb{Z}$, thus
                $$h^3(\mathbb{P}^r, \mathcal{I}_X(j))= \begin{cases} \frac{(j+1)(dj+2)}{2} &\mbox{ for all } j \leq - 2,\\
                0                     &\mbox{ for all } j \geq -1. \end{cases}$$

\end{itemize}
\end{theorem}
\begin{proof}
(a): The projecting scroll $\widetilde{X} \subset \mathbb{P}^{d+1}$ satisfies the conditions $N_{2,p}$ for all $p \in \mathbb{N}$, so that $\reg(\widetilde{X}) = 2$. Therefore, by Theorem~\ref{2.5' Theorem} we get $\reg(X) \leq (d-1)-r+2 =d-r+3$, and this proves our claim.\\
(b): As $X$ is a surface, we have  $\Reg(X) \subset \Nor(X) \subset \CM(X)$. Therefore, it suffices to show that $\CM(X) \subset \Reg(X).$\\
So, let $x \in \CM(X)$ be a closed point. We always write $\pi := \pi_\Lambda$. Then
$\big(\pi_*\mathcal{O}_{\widetilde{X}}\big)_x$ is a finite integral extension domain of the local $2$-dimensional CM ring $(\mathcal{O}_{X,x},\mathfrak{m}_{X,x})$. As the morphism $\pi: \widetilde{X} \rightarrow X$ is almost non-singular,
the finitely generated $\mathcal{O}_{X,x}$-module
$$\big((\pi_*\mathcal{O}_{\widetilde{X}})/ \mathcal{O}_X\big)_x = \big(\pi_*\mathcal{O}_{\widetilde{X}}\big)_x/ \mathcal{O}_{X,x}$$
is annihilated by some power of $\mathfrak{m}_{X,x}$ and hence contained in the local cohomology module $H^1_{\mathfrak{m}_{X,x}}(\mathcal{O}_{X,x})$. But this latter module vanishes because $\mathcal{O}_{X,x}$ is a local $\CM$-ring of dimension $>1$. This shows, that $x \notin \mathrm{Sing}(\pi)$. \\
But this means, that $x$ has a unique preimage $\widetilde{x} \in \widetilde{X}$ under the morphism $\pi$ and that
$$\mathcal{O}_{\widetilde{X},\widetilde{x}} \cong \mathcal{O}_{X,x}.$$
Assume now, that $x \notin \Reg(X).$ Then $\widetilde{x} \notin
\Reg(\widetilde{X})$. This means that $\widetilde{X} \subset \mathbb{P}^{d+1}$ is a singular $2$-fold scroll with vertex $\widetilde{x}$. But this implies that the tangent space $\mathrm{T}_{\widetilde{x}}(\widetilde{X})$ of $\widetilde{X}$ at $\widetilde{x}$ has dimension $d+1$. In view of the above isomorphism, we thus get the contradiction that the tangent space $\mathrm{T}_x(X)$ of $X \subset \mathbb{P}^r$ at $x$ has dimension $d+1$. This proves that indeed $x \in \Reg(X)$. \\
(c): Let $E$ denote the homogeneous coordinate ring of the projecting scroll $\widetilde{X} \subset \mathbb{P}^{d+1}$. As $E$ is a CM ring, we have canonical inclusions of graded rings
$$A \subset B(A)=\bigcup_{n\in \mathbb{N}}(A:_{\rm{Quot}(A)}\mathfrak{a}^n) \subset E$$ (see Notation and Reminder~\ref{2.7' Notation and Reminder} (A)). Keeping in mind statement (b) and observing that the projection morphism $\pi_\Lambda: \widetilde{X} \longrightarrow X$ provides the normalization of $X$, we thus get
$$\Proj(A/\mathfrak{a}) =  X \setminus \CM(X) = X \setminus \Nor(X) = \Sing(\pi_{\Lambda})$$
and hence $E \subset
\bigcup_{n\in \mathbb{N}}(A:_{\rm{Quot}(A)}\mathfrak{a}^n) = B(A)$. Therefore $E = B(A)$ and statement (c) is shown.\\
(d): Let $B := B(A)$ as above, let $D:= D_{A_+}(A)$ and consider the short exact sequence of graded $S$-modules
$$0 \longrightarrow D \longrightarrow B \longrightarrow C \longrightarrow 0.$$
Observe that $\dim(C) \leq 1$, $\depth(C) > 0$ and that
$\widetilde{C} \cong \mathcal{F} := \pi_*
\mathcal{O}_{\widetilde{X}}/\mathcal{O}_X$, so that
$$\dim_{\Bbbk}(D_{A_+}(C)_j) = \e(X) \mbox{ for all } j \in \mathbb{Z}.$$
As $B$ is a Cohen-Macaulay module of dimension $3$, we have
$$H^2_{*}(\mathbb{P}^r,\mathcal{I}_X) \cong H^2(A) \cong H^2(D) \cong H^1(C) \cong D_{A_+}(C)/C.$$
Hence, if $\e(X) = 0$, we have indeed $h^2(\mathbb{P}^r,\mathcal{I}_X(j)) = 0$ for all $j \in \mathbb{Z}$.\\

So, let $\e(X) > 0$. As $C_j = 0$ for all $j \leq 0$ we get that $h^2(\mathbb{P}^r,\mathcal{I}_X(j))=\e(X)$ for all $j \leq 0$. As $\dim_{\Bbbk}(D_1) = r+1 + h^1(\mathbb{P}^r,\mathcal{I}_X(1))$, we have $\dim_{\Bbbk}(C_1) = \dim_{\Bbbk}(B_1) - \dim_{\Bbbk}(D_1) = d+2 - \big(r+1+h^1(\mathbb{P}^r,\mathcal{I}_X(1))\big)$ and hence
\begin{align*}h^2(\mathbb{P}^r,\mathcal{I}_X(1))  &= \dim_\Bbbk(D_{A_+}(C)_1)-\dim_\Bbbk(C_1) = \e(X) - \dim_\Bbbk(C_1) = \\ &= \e(X) - h^1(\mathbb{P}^r,\mathcal{I}_X(1)) + r - d -1.\end{align*}
Observe that $C$ is a Cohen-Macaulay module of dimension $1$. Moreover, the regularity of $B$ as an $A$-module and as a $B$-module take the same value $\reg(B) = \reg(\widetilde{X})-1 = 1$, so that the $A$-module $B$ is generated in degree $1$. Therefore, the $A$-module $C$ is generated in degree $1$, and hence $\dim_{\Bbbk}(H^1(C)_j) \leq \max\{0, \dim_{\Bbbk}(H^1(C)_{j-1})-1\}$ for all $j > 1$ (see for example \cite{BS1}). Therefore we get indeed  $h^2(\mathbb{P}^r,\mathcal{I}_X(j)) \leq \max\{0, h^2(\mathbb{P}^r,\mathcal{I}_X(j-1)) -1 \}$ for all $j > 1.$\\
By statement (a) we have $h^2(\mathbb{P}^r,\mathcal{I}_X(j)) = 0$ for all $j \geq d-r+1$. \\
Finally, let us consider the exact sequence
$$0 \longrightarrow \mathcal{I}_X(-1) \longrightarrow \mathcal{I}_X \longrightarrow \mathcal{I}_{\mathcal{C}} \longrightarrow 0,$$
where $\mathcal{C} = X \cap \mathbb{H} \quad (\mathbb{H}\in
\mathbb{G}(r-1,\mathbb{P}^r))$ is a general hyperplane section of
$X$. As $\mathcal{C} \subset \mathbb{H} = \mathbb{P}^{r-1}$ is a
smooth rational curve of degree $d$, we have $h^1(\mathbb{P}^{r-1},
\mathcal{I}_{\mathcal{C}}(1)) = d-r+1$ and
$h^1(\mathbb{P}^{r-1},\mathcal{I}_{\mathcal{C}}(j+1)) \leq \mathrm{max}\{0, h^1(\mathbb{P}^{r-1},\mathcal{I}_{\mathcal{C}}(j))-1\}$ for all $j \geq 1$,
so that $h^1(\mathbb{P}^{r-1},\mathcal{I}_{\mathcal{C}}(d-r+1)) \leq 1.$ Applying the above exact sequence and keeping in mind that $h^2(\mathbb{P}^r,\mathcal{I}_X(d-r+1)) = 0$ we get $h^2(\mathbb{P}^r,\mathcal{I}_X(d-r)) \leq 1.$ and this proves our claim. \\
(e): Let the notation be as above. As the sheaf $\mathcal{F} :=
\pi_* \mathcal{O}_{\widetilde{X}}/\mathcal{O}_X$ has finite support,
the sequence
$$0 \longrightarrow \mathcal{O}_X \longrightarrow  \pi_* \mathcal{O}_{\widetilde{X}} \longrightarrow \mathcal{F}  \longrightarrow 0$$
together with the well known formulas for the cohomology of a rational surface scroll yields that
$$h^3(\mathbb{P}^r, \mathcal{I}_X(j)) = h^2(X,\mathcal{O}_X(j))= h^2(\widetilde{X},\mathcal{O}_{\widetilde{X}}(j)) = \begin{cases} \frac{(j+1)(dj+2)}{2}, &\mbox{ if } j \leq - 2,\\
                0,                     &\mbox{ if } j \geq -1. \end{cases}$$
This proves statement (e).
\end{proof}

\subsection*{Local properties} Finally, we want to give the following result, in which $\mathrm{mult}_z(Z)$ is used to denote the \textit{multiplicity} of the noetherian scheme $Z$ at the point $z \in Z.$

\begin{proposition}
\label{prop:loc,properties} Let $X \subset \mathbb{P}^r$ be a non-degenerate irreducible projective and sectionally smooth rational surface of degree $d$ with $r \geq 4$ and $d \geq r+2$. Let $\pi = \pi_\Lambda: \widetilde{X} \twoheadrightarrow X$ be the standard normalization of $X$. Then it holds:
\begin{itemize}
\item[\rm{(a)}] If $x \in \mathrm{Sing}(X)$, then $2 \leq \mathrm{length}(\pi^{-1}(x)) \leq \mathrm{max}\{\mathrm{mult}_x(X),\mathrm{e}_x(X)+1\}.$
\item[\rm{(b)}] If $\mathbb{K} \in \mathbb{G}(k,\mathbb{P}^r)$ with $0 \leq k \leq r-1$ and $\mathrm{dim}(X \cap \mathbb{K}) \leq 0$, then
$$\mathrm{length}\big(\Reg(X) \cap \mathbb{K}\big) + 2\#\big(\Sing(X) \cap \mathbb{K}\big) \leq d-r+k+2.$$
\item[\rm{(c)}] If $\mathbb{L} \in \mathbb{G}(1,\mathbb{P}^r)$ is a proper extremal secant line to $X$ with $X \cap \mathbb{L} \subset \mathrm{Reg}(X)$, and $x \in \mathrm{Sing}(X)$, then $\dim(X \cap \langle x,\mathbb{L}\rangle) = 1$.
\end{itemize}
\end{proposition}
\begin{proof} (a): Indeed, as $x$ is an isolated point of $\mathrm{Sing}(X) = \mathrm{Sing}(\pi_\Lambda)$ and as $\widetilde{X}$ is a Cohen-Macaulay surface, the semilocal ring $\mathcal{O}_{\widetilde{X},x} := \big({\pi}_{*}\mathcal{O}_{\widetilde{X},x}\big)_x$ is a Cohen-Macaulay finite integral extension domain of the local ring $\mathcal{O}_{X,x}$ and $H^1_{\mathfrak{m}_{X,x}}(\mathcal{O}_{X,x}) \cong \mathcal{O}_{\widetilde{X},x}/\mathcal{O}_{X,x}$. As $\pi^{-1}(x) \cong \mathcal{O}_{\widetilde{X},x}/\mathfrak{m}_{X,x}\mathcal{O}_{\widetilde{X},x}$ our claim follows easily. \\
(b): Let $\mathbb{K}' := \overline{(\pi'_\Lambda)^{-1}(\mathbb{K})} \in \mathbb{G}(d-r+k+1,\mathbb{P}^{d+1})$ be the closed preimage of $\mathbb{K}$ under the linear projection $\pi'_\Lambda:\mathbb{P}^{d+1} \setminus \Lambda \twoheadrightarrow \mathbb{P}^r$. Then $\widetilde{X} \cap \mathbb{K}' = \pi^{-1}(X \cap \mathbb{K})$, and so
the morphism $\pi: \widetilde{X} \twoheadrightarrow X$ induces an isomorphism
$$\pi\upharpoonright: \widetilde{X} \cap \mathbb{K}'\setminus \big[\pi^{-1}\big(\mathrm{Sing}(\pi)\cap \mathbb{K}\big)\big] \stackrel{\cong}{\longrightarrow} X \cap \mathbb{K} \setminus \big(\mathrm{Sing}(\pi)\cap \mathbb{K}\big) = \mathrm{Reg}(X) \cap \mathbb{K}.$$
Moreover $\widetilde{X} \cap \mathbb{K}'$ is of dimension $\leq 0.$ As $\widetilde{X} \subseteq \mathbb{P}^{d+1}$ is a surface scroll and $\mathbb{K} \in \mathbb{G}(d-r+k+1,\mathbb{P}^{d+1})$, we get that $\mathrm{length}(\widetilde{X} \cap \mathbb{K}') \leq d-r+k+2.$ Now, our statement follows by the first inequality of statement (a).\\
(c): Observe, that $\mathbb{K} := \langle x,\mathbb{L}\rangle$ is a plane. Assume that $\dim(X\cap \mathbb{K}) \leq 0$. Then, by statement (b) we get
$$d-r+5 \leq \mathrm{length}(X \cap \mathbb{L}) + 2 \leq \mathrm{length}\big(\mathrm{Reg}(X) \cap \mathbb{K}\big) + 2\#\big(\mathrm{Sing}(X)\cap \mathbb{K}\big)
\leq d-r+4.$$
This contradiction proves our claim.
\end{proof}

\section{Surfaces of Type I}

\subsection*{The Betti numbers} In this section we study the surfaces which fall under type I of our classification. We begin by investigating their Betti numbers.

\begin{convention and remark}
\label{convention and remark}
(A) Let $X \subset \mathbb{P}^5$ denote a projective surface contained in a smooth rational three-fold scroll in $\mathbb{P}^5$, hence that
$$X \subset W:= S(1,1,1) \subset \mathbb{P}^5.$$
We assume  furthermore that the divisor $X \subset W$ satisfies
$$X \approx H+(d-3)F \mbox{ for some } d \geq 5,$$
where $H$ is a hyperplane section and $F$ is a plane of $W$. Then it is easy to check that
\begin{equation*}
 \mathrm{deg} (X)=d \quad \mbox{and} \quad \mathrm{reg} (X)=d-2.
\end{equation*}

(B) With the definition of the Betti numbers and the Betti diagram we follow the notations suggested by D. Eisenbud (see \cite{E}). So, if $Z \subset \mathbb{P}^r$ is a closed subscheme, with homogeneous vanishing ideal $I_Z \subset S:= \Bbbk[x_0,x_1,\ldots,x_r]$ and homogeneous coordinate ring $A_Z := S/I_Z$, we write
$$\beta_{i,j} =  \beta_{i,j}(Z) := \mathrm{dim}_\Bbbk
\big(\mathrm{Tor}^S_i(\Bbbk,A_Z)_{i+j}\big) \mbox{ for all } i \in \mathbb{N}_0 \mbox{ and all } j \in\mathbb{Z}.$$
As usually, if $Z$ is non-degenerate, we list this numbers only the range $1\leq i \leq r+1-\mathrm{depth}(Z)$ and $1\leq j < \mathrm{reg}(Z)$
\end{convention and remark}

\begin{theorem} \label{t1-betti}
In the previous notation we have the following Betti diagram  of $X$
\[
\begin{tabular}{|c|c|c|c|c|c|c|}\hline
                          \multicolumn{2}{|c||}{$i$}& $1$&$2$&$3$&$4$&$5$  \\\cline{1-7}
             \multicolumn{2}{|c||}{$\beta_{i,1}$}& $3$&$2$&$0$&$0$&$0$ \\\cline{1-1}
             \multicolumn{2}{|c||}{$\beta_{i,2}$}& $0$& $0$&$0$& $0$&$0$  \\\cline{1-1}
             \multicolumn{2}{|c||}{$\vdots$}& $\vdots$&$\vdots$&$\vdots$&$\vdots$&$\vdots$\\\cline{1-1}
             \multicolumn{2}{|c||}{$\beta_{i,d-4}$}& $0$&$0$&$0$&$0$&$0$\\\cline{1-1}
             \multicolumn{2}{|c||}{$\beta_{i,d-3}$}& $\beta_{1,d-3}$&$\beta_{2,d-3}$&$\beta_{3,d-3}$&$\beta_{4,d-3}$&$\beta_{5,d-3}$\\\cline{1-7}
             \end{tabular}
\]
with the following entries
\begin{gather*}
 \beta_{1,d-3}={{d-1}\choose{2}},\quad \beta_{2,d-3}=2(d-1)(d-3),\quad\beta_{3,d-3}=3(d^2-5d+5) \\
 \beta_{4,d-3}=2(d-2)(d-4)\quad \mbox{and} \quad \beta_{5,d-3}={{d-3}\choose{2}}.
\end{gather*}
\end{theorem}

\begin{proof} As in Convention and Remark~\ref{convention and remark} (A), we put $W = S(1,1,1)$, and denote the coordinate rings of
$X$ and $W$ by $A_X$ and $A_W$ respectively. Then there is a short
exact sequence
\[
0 \to I_X/I_W \to A_W \to A_X \to 0
\]
of graded $S$-modules,
where $I_W$ denotes the defining ideal of $W \subset \mathbb{P}^5$. In a
first step we compute the Hilbert series
\[
H(I_X/I_W,t) = \sum_{n \in \mathbb{Z}}\dim_{\Bbbk}[I_X/I_W]_n \cdot t^n
\]
of $I_X/I_W$. By applying sheaf cohomology to the corresponding short
exact sequence $0 \to \mathcal{I}_X/\mathcal{I}_W \to \mathcal{O}_W
\to \mathcal{O}_X \to 0$ we obtain an isomorphism
\[
[I_X/I_W]_n \cong H^0(W,\mathcal{O}_W(-X +nH))
\cong H^0(W,\mathcal{O}_W((n-1)H - (d-3)F)).
\]
Therefore, it follows that
\[
\dim_{\Bbbk}[I_X/I_W]_n = \begin{cases}
\binom{n+1}{2} (n-d+3)  & \text{ if } n \geq d-2 \\
\quad \quad  0 & \text{ if } n < d-2
\end{cases}
\]
As an application to the Hilbert series $H(I_X/I_W,t)$ it turns out that
\[
H(I_X/I_W, t) = \sum_{n \geq d-2}\binom{n+1}{2}(n-d+3) t^n =
\frac{t^{d-2}}{(1-t)^4} \big( \binom{d-1}{2} -(d-1)(d-4)t +
\binom{d-3}{2}t^2 \big).
\]
The formula for the expression of the generating function as a rational
function might be proven directly or by some Computer Algebra System.
The Hilbert series of $A_Y$ is given by $H(A_W,t) = (1+2t)/(1-t)^4$. By the
above short exact sequence of graded modules it follows that
\[
H(A_X,t) = 
\frac{1}{(1-t)^4} \big( 1+2t - \binom{d-1}{2}t^{d-2}
+(d-1)(d-4)t^{d-1} - \binom{d-3}{2}t^d \big).
\]
As a consequence of \cite[Remark 4.8 (2)]{P} the
Betti diagram of $X$ now must have the shape as indicated in the statement. In
particular the first row has the stated form. For the sake of simplicity, we
put $\beta_{i,d-3} = \beta_i, i = 1,\ldots,5$. Then by the additivity of
the Hilbert series on short exact sequences of graded $S$-modules the Betti
diagram implies the following form of the Hilbert series $H(A_X,t)$
\[
H(A_X,t) = \frac{1}{(1-t)^6} \big(
1-3t^2-\beta_1t^{d-2}+2t^3+\beta_2t^{d-1} -\beta_3 t^d+
\beta_4t^{d+1}-\beta_5 t^{d+2} \big).
\]
By comparing both expressions for our Hilbert series we obtain
the desired values for the remaining Betti numbers.
\end{proof}

\noindent

\subsection*{Cohomological properties} We now provide a result which
summarizes some cohomological properties of surfaces of type I. It is worth noticing that in this case the values for
the sheaf cohomology of $\mathcal{O}_X$ and $\mathcal{I}_X$, and hence also the \textit{index of normality}
$$N(X) := \sup\{j\in \mathbb{N} \mid h^1(\mathbb{P}^r,\mathcal{I}_X(j)) \neq 0 \}\quad ( \in \mathbb{N} \cup \{-\infty \})$$
of $X$ are completely determined.

\begin{theorem} \label{t1-coh}
With the previous notation there are the following
equalities for the cohomology:
\begin{itemize}
\item[(a)] $h^1(\mathbb{P}^5, \mathcal{I}_X(j)) = \binom{j+1}{2} (d-j-3)$
for $1 \leq j \leq d-4$ and zero else.
\item[(b)] $h^0(X,\mathcal{O}_X(j))  = 1/2(j+1)(dj+2)$ for all $j \geq 0$ and zero else.
\item[(c)] $h^1(X,\mathcal{O}_X(j))  = 0$ for all $j \in \mathbb{Z}$.
\item[(d)] $h^2(X, \mathcal{O}_X(-j)) = 1/2(j-1)(dj-2)$ for all $j \geq 2$ and zero else.
\end{itemize}
\end{theorem}

\begin{proof} We start with the proof of (a), using the notations introduced in Convention and Remark~\ref{convention and remark}. Clearly
$H^1(\mathbb{P}^5, \mathcal{I}_X(j)) = 0$ for all $j \leq 0$. So let
$j \geq 1$.  Then we use the short
exact sequence
\[
0 \to \mathcal{I}_W \to \mathcal{I}_X \to \mathcal{O}_W(-X) \to 0.
\]
Since $W$ is arithmetically Cohen-Macaulay, this sequence yields that $H^i(\mathbb{P}^5, \mathcal{I}_W(j)) = 0$ for
$i = 1,2$ and all $j \in \mathbb{Z}$. Therefore the long exact cohomology
sequence induces isomorphisms $H^1(\mathbb{P}^5,\mathcal{I}_X(j)) \cong
H^1(W,\mathcal{O}_W(jH -X))$ for all $j \in \mathbb{Z}$. Because of
$X \approx H + (d-3)F$ it follows that
\[
h^1(\mathbb{P}^5, \mathcal{I}_X(j)) = h^1(W,\mathcal{O}_W((j-1)H- (d-3)F)
= \binom{j+1}{2} h^1(\mathbb{P}^1, \mathcal{O}_{\mathbb{P}^1}(j+2-d)).
\]
By duality we get $h^1(\mathbb{P}^1, \mathcal{O}_{\mathbb{P}^1}(j+2-d)) =
h^0(\mathbb{P}^1, \mathcal{O}_{\mathbb{P}^1}(d-j-4))$. This proves
statement (a).\\
Because of $h^1(\mathbb{P}^5, \mathcal{I}_X(1)) = d - 4$ (as shown
in (a)) and because of $\deg(X) = d$, the linearly normal embedding
of $X$ implies that $X \subset \mathbb{P}^5$ is isomorphic to the
linear projection of a smooth rational normal surface scroll
$\tilde{X} \subset \mathbb{P}^{d+1}$. As a consequence we have
$H^i(X,\mathcal{O}_X(j)) \cong
H^i(\tilde{X},\mathcal{O}_{\tilde{X}}(j))$ for all $i,j \in
\mathbb{Z}$. Since $\tilde{X}$ is arithmetically Cohen-Macaulay this
yields
statement in (c). \\
The Hilbert function
\[
j \mapsto h_{A_{\tilde{X}}}(j) := \dim_{\Bbbk}
[A_{\tilde{X}}]_j = h^0(\tilde{X},\mathcal{O}_{\tilde{X}}(j))
\]
of the coordinate ring $A_{\tilde{X}}$
is given by $1/2(j+1)(dj+2)$, and this proves statement (b). \\
By interchanging $j$ and $-j$ this provides also the proof of the statement in (d).
\end{proof}

\subsection*{The extremal secant locus} Now, we consider the (special) secant locus of a surface of type I. We first give the following auxiliary result, which shall be of use for us again later.

\begin{lemma}\label{4.12'' Lemma}
Let $s > 1$, let $\mathcal{C} \subset \mathbb{P}^s$ be a closed subscheme of dimension $1$ and degree $d$ and let $\mathbb{H} = \mathbb{P}^{s-1} \subset \mathbb{P}^s$ be a hyperplane. Then
$$\mathrm{length}(\mathcal{C} \cap \mathbb{H}) \geq d \mbox{  with equality if and only if } {\rm Ass}_{\mathcal{C}}(\mathcal{O}_\mathcal{C}) \cap \mathbb{H} = \emptyset.$$
\end{lemma}
\begin{proof} Let $R = \Bbbk\oplus R_1\oplus R_2\oplus \ldots = \Bbbk[R_1]$ be the homogeneous coordinate ring of $\mathcal{C}$ and let $f \in R_1$ be such that $\mathcal{C}\cap \mathbb{H} = \Proj(R/fR)$. Let $H_R(t) = dt + c$ be the Hilbert polynomial of $R$. Then, the two exact sequences
$$0\rightarrow fR \rightarrow R \rightarrow R/fR\rightarrow 0 \mbox{ and } 0 \rightarrow(0:_R f)(-1) \rightarrow R(-1) \rightarrow fR \rightarrow 0$$
yield that the Hilbert polynomial of $R/fR$ is given by
$$H_{R/fR}(t) = d + H_{(0:_R f)}(t-1).$$
Observe that the polynomial $H_{(0:_R f)}(t-1)$ vanishes if and only if $(0:_R f)_t = 0$ for all $t \gg 0$, hence if and only if
$$f\notin \bigcup_{\mathfrak{p} \in {\rm Ass}(R)\setminus \{R_+\}} \mathfrak{p}.$$
But this latter condition is equivalent to the requirement that ${\rm Ass}(\mathcal{C}) \cap \mathbb{H} = \emptyset$.
\end{proof}

\begin{proposition}
\label{prop:extseclocI}  Let $X \subset \mathbb{P}^5$ be a surface of maximal sectional regularity of degree $d > 5$ which is of type I. Then, in the notations of Convention and Remark~\ref{convention and remark} we have:
\begin{itemize}
\item[\rm{(a)}] $\mathbb{F}^+(X) = \mathbb{F}(X) = W = S(1,1,1).$
\item[\rm{(b)}] ${}^*\Sigma(X) = \Sigma(X).$
\item[\rm{(c)}] The image $\psi\big(\Sigma(X)\big)$ of $\Sigma(X)$ under the Pl\"ucker embedding
                $\psi: \mathbb{G}(1,\mathbb{P}^5) \rightarrow \mathbb{P}^{14}$ is a Veronese surface in a subspace $\mathbb{P}^5 \subset \mathbb{P}^{14}$.
\end{itemize}
\end{proposition}
\begin{proof} Statement (a) is a restatement of Notation and Remark~\ref{4.2'' Notation and Reminder} (C)(a).\\
(b), (c): We identify $S(1,1,1) = W$ with the image of the Segre embedding $\sigma:\mathbb{P}^1 \times \mathbb{P}^2 \rightarrow \mathbb{P}^5$. Consider the canonical projection
$$\varphi:\mathbb{P}^1 \times \mathbb{P}^2 \twoheadrightarrow \mathbb{P}^1 \mbox{ and its restriction } \varphi\upharpoonright:X \twoheadrightarrow \mathbb{P}^1.$$
Let
$$\Theta := \{\mathbb{P}^1 \times \{q\} \mid q \in \mathbb{P}^2\} \subset \mathbb{G}(1,\mathbb{P}^5)$$
denote the closed subset of all fibers under the canonical projection $\mathbb{P}^1 \times \mathbb{P}^2 \twoheadrightarrow \mathbb{P}^2$,
hence the set of all line sections of $\varphi.$ \\
Fix a closed point $p \in \mathbb{P}^1$. Then, the fiber $\varphi^{-1}(p) = \{p\}\times \mathbb{P}^2 =: \mathbb{P}_p^2$ is a ruling plane of $W$.
As $X$ is smooth (see Corollary~\ref{corollary 2.1'}) and hence locally Cohen-Macaulay,
the fiber $(\varphi\upharpoonright)^{-1}(p) = X \cap \varphi^{-1}(p) = X \cap \mathbb{P}_p^2$ is of pure dimension $1$ and has no closed associated points.
Therefore $\mathrm{length}(X \cap \mathbb{L}) = \deg(X \cap \mathbb{P}^2_p)$ for all lines $\mathbb{L} \subset \mathbb{P}^2_p$
not contained in $X$ (see Lemma \ref{4.12'' Lemma}). Consequently, if $\mathbb{P}^2_p$ would contain a proper extremal secant line to $X$,
the curve $X \cap \mathbb{P}_p^2 \subset \mathbb{P}^2_p$ would be pure and of degree $d-2$, so that $X$ would be of type II.
This contradiction shows, that no proper extremal secant line to $X$ is contained in a ruling plane $\mathbb{P}^2_p$.
Hence each proper secant line to $X$ must be a line section of $W$.\\
As $\Theta \subset \mathbb{G}(1,\mathbb{P}^5)$ is closed, it follows that ${}^*\Sigma(X) \subseteq \Sigma(X) \subseteq \Theta,$ so that finally
$\psi\big({}^*\Sigma(X)\big) \subseteq \psi\big(\Sigma(X)) \subseteq \psi(\Theta).$ Standard arguments on Pl\"ucker embeddings show that $\psi(\Theta)$
is the Veronese surface in some subspace $\mathbb{P}^5 \subset \mathbb{P}^{14}$.
As ${}^*\Sigma(X)$ is of dimension $2$ (see Proposition~\ref{prop:dimspecextrloc}), as $\psi(\Theta)$ is irreducible and as $\psi$ is
a closed embedding, statements (c) and (b) follow.
\end{proof}

\section{Surfaces of Type II}

\subsection*{The cohomological aspect} In this section, we investigate the surfaces of maximal sectional regularity which fall under type II. For the whole section we make the following convention.

\begin{convention and notation}
\label{convention and notation} Let $5 \leq r < d$ and let $X \subset \mathbb{P}^r$ be a surface of degree $d$ and of maximal sectional regularity of type II which is not a cone. Set $Y := X \cup \mathbb{F}$, where $\mathbb{F} = \mathbb{F}(X) = \mathbb{P}^2$ denotes the extremal plane of $X$. Moreover, let $I$ and $L$ respectively denote the homogeneous vanishing ideal of $X$ and of $\mathbb{F}$ in $S = \Bbbk[x_0,x_1,\ldots,x_r]$.
\end{convention and notation}

\begin{theorem}\label{4.14'' Theorem} Let the notations and hypotheses be as in Convention and Notation~\ref{convention and notation}. Then the following statements hold
\begin{itemize}
\item[\rm{(a)}] \begin{itemize}
                \item[\rm{(1)}] $\mathrm{reg}(X) = d-r+3$ and $\mathrm{e}(X) \geq \binom{d-r+2}{2}.$
                \item[\rm{(2)}] $h^1(\mathbb{P}^r, \mathcal{I}_X(j)) = 0$ for all $j\leq 1.$ In particular $X$ is linearly normal.
                \item[\rm{(3)}] $h^2(\mathbb{P}^r,\mathcal{I}_X(j)) \begin{cases} = \e(X)                                                  &\mbox{ for all } j \leq 0;\\
                                                                  = \e(X) + r - d -1                                                       &\mbox{ for } j=1;\\
                                                                  \leq \max\{0, h^2(\mathbb{P}^r,\mathcal{I}_X(j-1)) -1 \}                 &\mbox{ for all } j > 1;\\
                                                                  = 1                                                                      &\mbox{ for } j = d-r;\\
                                                                  = 0                                                                      &\mbox{ for all } j \geq d-r+1.
                                                    \end{cases}$
                \item[\rm{(4)}] $h^3(\mathbb{P}^r, \mathcal{I}_X(j))= \begin{cases} \frac{(j+1)(dj+2)}{2}, &\mbox{ if } j \leq - 2,\\
                                0,                     &\mbox{ if } j \geq -1. \end{cases}$
                \end{itemize}
\item[\rm{(b)}] \begin{itemize}
                \item[\rm{(1)}] $\mathrm{reg}(Y) \leq d-r+3.$
                \item[\rm{(2)}] $H^1_{*}(\mathbb{P}^r, \mathcal{I}_Y) \cong H^1_{*}(\mathbb{P}^r,\mathcal{I}_X).$
                \item[\rm{(3)}] $h^2(\mathbb{P}^r,\mathcal{I}_Y(j)) = h^2(\mathbb{P}^r,\mathcal{I}_X(j)) - \max\{0,\binom{-j+d-r+2}{2}\}$ for all $j \geq 0$. In particular $h^2(\mathbb{P}^r,\mathcal{I}_Y(1)) = \e(X)-\binom{d-r+2}{2}$ and $h^2(\mathbb{P}^r,\mathcal{I}_Y(d-r)) = 0.$
                \item[\rm{(4)}] $h^2(\mathbb{P}^r,\mathcal{I}_Y(j)) \geq h^2(\mathbb{P}^r,\mathcal{I}_Y(j-1))$ for all $j \leq 1$, with equality for $j = 1$.
                \item[\rm{(5)}] $h^2(\mathbb{P}^r,\mathcal{I}_Y(j)) \leq {\rm max}\{0,h^2(\mathbb{P}^r,\mathcal{I}_Y(j-1)) - 1\}$ for all $j > 1$.
                \item[\rm{(6)}] $h^3(\mathbb{P}^r,\mathcal{I}_Y(j)) = 0$ for all $j \geq 0.$
                \item[\rm{(7)}] If $h^2(\mathbb{P}^r, \mathcal{I}_Y) = 0$, then $h^2(\mathbb{P}^r,\mathcal{I}_Y(j)) = 0$ for all $j \in \mathbb{Z}$.
                \end{itemize}
\item[\rm{(c)}] For the pair $\tau(X) := \big(\depth(X),\depth(Y)\big)$ we have
                \begin{itemize}
                \item[\rm{(1)}] $\tau(X) = (2,3)$ if $r+1 \leq d \leq 2r-4$;
                \item[\rm{(2)}] $\tau(X) \in \{(1,1),(2,2),(2,3)\}$ if $2r-3 \leq d \leq 3r-7$;
                \item[\rm{(3)}] $\tau(X) \in \{(1,1),(2,2)\}$ if $3r-6 \leq d$.
                \end{itemize}
\item[\rm{(d)}] $h^0(\mathbb{P}^r,\mathcal{I}_X(2))\geq \binom{r}{2}-d-1$ with equality if and only if $\tau(X) = (2,3).$
\end{itemize}
\end{theorem}

As an immediate application, we get the following information on the \textit{cohomology tables}
$$\big(h^i(\mathbb{P}^r,\mathcal{I}_X(j))\big)_{i =1,2,3 \mbox{ and } j \in \mathbb{Z}}, \quad \quad \big(h^i(\mathbb{P}^r,\mathcal{I}_Y(j))\big)_{i =1,2,3 \mbox{ and } j \in \mathbb{Z}}$$
of the sheaves of vanishing ideals  $\mathcal{I}_X, \mathcal{I}_Y \subset \mathcal{O}_{\mathbb{P}°r}$ of $X$ and $Y$.

\begin{corollary}\label{coro:cohomology}
Let $X$ and $Y$ be as above. Then the ideal sheaves $\mathcal{I}_X, \mathcal{I}_Y \subset \mathcal{O}_{\mathbb{P}°r}$ of $X$ and $Y$ have the following cohomology tables:
\begin{equation*}
\begin{tabular}{| c | c | c | c | c | c | c |c | c | c | c | c |c |}
    \hline
    $j$ & $\cdots$ & $-2$
    & $-1$ & $0$ & $1$ & $2$ &$\cdots$ &$\varkappa-1$ &$\varkappa$ &$\varkappa+1$ &$\varkappa+2$ & $\cdots$ \\ \hline
    $h^1(\mathbb{P}^r,\mathcal{I}_X(j))$ &$\cdots$ &$0$ &$0$ &$0$ &$0$ &$\ast$ &$\cdots$ &$\ast$ &$\ast$ &$\ast$ &$0$ & $\cdots$\\ \hline
    $h^2(\mathbb{P}^r,\mathcal{I}_X(j))$ &$\cdots$ &$\e$ &$\e$ &$\e$ &$\e-\varkappa-1$ &$\ast$ &$\cdots$ &$\ast$ &$1$ &$0$ &$0$ & $\cdots$\\ \hline
    $h^3(\mathbb{P}^r,\mathcal{I}_X(j))$ &$\cdots$ &$\ast$ &$0$ &$0$ &$0$ &$0$ &$\cdots$ &$0$ &$0$ &$0$ &$0$ & $\cdots$\\ \hline
    \end{tabular}
\end{equation*}
and
\begin{equation*}
\begin{tabular}{| c | c | c | c | c | c | c |c | c | c | c | c |c |}
    \hline
    $j$ & $\cdots$ & $-2$
    & $-1$ & $0$ & $1$ & $2$ &$\cdots$ &$\varkappa-1$ &$\varkappa$ &$\varkappa+1$ &$\varkappa+2$ & $\cdots$ \\ \hline
    $h^1(\mathbb{P}^r,\mathcal{I}_Y (j))$ &$\cdots$ &$0$ &$0$ &$0$ &$0$ &$\ast$ &$\cdots$ &$\ast$ &$\ast$ &$\ast$ &$0$ & $\cdots$\\ \hline
    $h^2(\mathbb{P}^r,\mathcal{I}_Y (j))$ &$\cdots$ &$\e$ &$\e$ &$\e$ &$\e-\binom{d-r+2}{2}$ &$\ast$ &$\cdots$ &$\ast$ &$0$ &$0$ &$0$ & $\cdots$\\ \hline
    $h^3(\mathbb{P}^r,\mathcal{I}_Y (j))$ &$\cdots$ &$\ast$ &$\ast$ &$0$ &$0$ &$0$ &$\cdots$ &$0$ &$0$ &$0$ &$0$ & $\cdots$\\ \hline
    \end{tabular}
\end{equation*}
where $\varkappa := d-r, \mathrm{e} :=\mathrm{e}(X)$ and $\ast$ stands for non-specified non-negative integers.
\end{corollary}

\subsection*{An auxiliary result} Before we establish Theorem~\ref{4.14'' Theorem} we prove the following Lemma.

\begin{lemma}
\label{4.12'' Lemma+}
Let the notations and hypotheses be as in Convention and Notation~\ref{convention and notation}. In addition let
$\mathcal{C} := X \cap \mathbb{F}.$ Then the following statements hold.
\begin{itemize}
\item[\rm{(a)}] Each line $\mathbb{L} \subset \mathbb{F}$ which is not contained in $X$, satisfies
$$\mathrm{length}(\mathcal{C}\cap \mathbb{L}) = \mathrm{length}(X\cap\mathbb{L}) = d-r+3.$$
In particular, $\mathcal{C} \subset \mathbb{F}$ is a curve of degree $d-r+2$ and has no closed associated points.
\item[\rm{(b)}] $I_{d-r+3} \setminus I \cap L \neq \emptyset$ and for each $f \in I_{d-r+3} \setminus I \cap L$ it holds $I = (I \cap L, f)$
\end{itemize}
\end{lemma}
\begin{proof}
(a): First let $h \in \mathbb{U}(X)$. Then $\mathbb{L}_h \subset \mathbb{F}$ and
\begin{equation*}
\mathrm{length}(\mathcal{C}\cap \mathbb{L}_h) = \mathrm{length}(X \cap \mathbb{L}_h) = \mathrm{length}(\mathcal{C}_h \cap \mathbb{L}_h) = d-r+3 .
\end{equation*}
This shows, that $\mathcal{C} \subset \mathbb{F}$ is a closed subscheme of dimension $1$ and degree $d-r+3$. Now, let $\mathbb{L} \subset
\mathbb{F}$ be an arbitrary line which is not contained in $X$. As $\mathcal{C} \subset \mathbb{F}$ is of dimension $1$ and of degree $d-r+3$ we have $\mathrm{length}(X \cap \mathbb{L}) =
\mathrm{length}(C \cap \mathbb{L}) \geq d-r+3$. As $\reg(X) \leq d-r+3$ (see Theorem \ref{theorem 2.9} (a)) we also have $\mathrm{length}(X \cap \mathbb{L}) \leq d-r+3$, so that indeed
$\mathrm{length}(X \cap \mathbb{L}) = d-r+3$. Now, it follows by Lemma \ref{4.12'' Lemma} that $\mathcal{C}$ has no closed associated point.  \\
(b): According to statement (a), there is a homogeneous polynomial $g \in S_{d-r+3} \setminus L$ such that the homogeneous vanishing ideal $(I + L)^{\sat} \subset S$ of
$\mathcal{C}$ in $S$ can be written as $(L,g)$. In particular we have $I_{\leq d-r+2} \subset L$. As $\reg(X) = d-r+3$, the ideal $I \subset S$ is generated by homogeneous
polynomials of degree $\leq d-r+3$. As $g \notin L$ it follows that $I_{d-r+3}$ is not contained in $L$ and hence that $I+L = (L,f)$ for all $f \in I_{d-r+3} \setminus
I \cap L$. Therefore $ I = I\cap(I+L) = I\cap (L,f) = (I\cap L,f)$ for all such $f$.
\end{proof}

\subsection*{Proof of Theorem \ref{4.14'' Theorem}.}
(a)(1): Since $X$ admits $(d-r+3)$-secant lines, we have $\reg (X) \geq d-r+3$. On the other hand, $\reg (X) \leq d-r+3$ by Theorem \ref{theorem 2.9} (a). This proves the stated equality for the regularity. For the moment, we postpone the proof of the stated estimate for the invariant $\mathrm{e}(X)$.  \\
(a)(2): It is obvious that $h^1(\mathbb{P}^r, \mathcal{I}_X(j)) = 0$ for all $j \leq 0$. Assume $h^1(\mathbb{P}^r, \mathcal{I}_X(1))> 0$. Then $X$ is a regular projection of a surface $X' \subset \mathbb{P}^{r+1}$. Note that $X'$ is again a sectionally smooth rational surface and hence $\reg(X') \leq d - r+2$ by Theorem \ref{theorem 2.9}(a). On the other hand, the preimage $\mathcal{C}'$ of $\mathcal{C} = X \cap \mathbb{F}$ under this regular projection is a plane curve of degree $(d-r+3)$, and hence $\reg(X') \geq d - r+3$. This contradiction proves our claim.\\
(a)(4): See Theorem \ref{theorem 2.9}(e)\\
(b)(4) and (b)(6): Let $h \in \mathbb{U}(X)$ and consider the induced exact sequence
\begin{equation*}
0 \rightarrow \mathcal{I}_Y (-1)  \rightarrow \mathcal{I}_Y \rightarrow \mathcal{I}_{\mathcal{C}_h \cup \mathbb{L}_h } \rightarrow 0
\end{equation*}
Keep in mind that $H^1 (\mathbb{P}^r , \mathcal{I}_{\mathcal{C}_h \cup \mathbb{L}_h } (j))=0$ for all $j \leq 1$ and $H^2 (\mathbb{P}^r ,  \mathcal{I}_{\mathcal{C}_h \cup \mathbb{L}_h } (j))=0$ for all $j \geq 1$ by \cite[Proposition 2.7(c),(d)]{BS2}. Both claims now follow easily.\\
(b)(5): Assume again that $h \in \mathbb{U}(X)$ and keep in mind that $S/(I\cap L ,h)^{\rm sat}$ is the homogeneous coordinate ring of $\mathcal{C}_h \cup \mathbb{L}_h$ in $S$. By \cite[Remark 3.2 B)]{BS2} the graded $S$-module
$$H^1(S/(I\cap L,h)^{\rm sat}) = \bigoplus_{j \in\mathbb{Z}}
H^1(\mathbb{P}^r,\mathcal{I}_{\mathcal{C}_h \cup \mathbb{L}_h}(j))$$ is generated by homogeneous elements of degree $2$. Now, the induced exact sequences of local cohomology modules
\begin{equation*}
H^1(S/(I \cap L, h)^{\rm sat}) \longrightarrow H^2(S/I\cap L)(-1)
\stackrel{\times h}{\longrightarrow} H^2(S/I\cap L) \longrightarrow
H^2(S/(I \cap L, h)^{\rm sat})
\end{equation*}
proves claim (b)(5), since the multiplication map $\cdot h$ is an epimorphism in all positive degrees and its kernel is generated by homogeneous elements of degree $2$. \\
(b)(7): This is an immediate consequence of (b)(4) and (b)(5).\\
(b)(2): Keep in mind that
\begin{equation*}
H^i_{*}(\mathbb{P}^r,\mathcal{I}_X) \cong H^i(S/I) \mbox { and } H^i_{*}(\mathbb{P}^r,\mathcal{I}_Y) \cong H^i(S/I\cap L) \mbox{ for } i = 1,2,3.
\end{equation*}
According to Lemma \ref{4.12'' Lemma+}(b) we have an exact sequence
\begin{equation*}
0 \rightarrow (S/L)(-d+r-3) \rightarrow S/I \cap L \rightarrow S/I \rightarrow 0.
\end{equation*}
Claim (b)(2) now follows immediately, since $H^i((S/L)(-d+r-3))$ vanishes for $i=1,2$.\\
(a)(3): For all $j \neq d-r$ the values of $h^2(\mathbb{P}^r,\mathcal{I}_X(j))$ are as requested by statement (a)(2) and by Theorem \ref{theorem 2.9} (d).
Observe that $h^2 (\mathbb{P}^r , \mathcal{I}_X (d-r)) \leq 1$ by Theorem \ref{theorem 2.9} (d). So, it remains to show that $H^2 (\mathbb{P}^r , \mathcal{I}_X (d-r)) \neq 0$. This follows from the exact sequence
\begin{equation*}
H^2(S/I)_{d-r} \longrightarrow H^3(S/L)_{-3} \cong \Bbbk \longrightarrow H^3(S/I\cap L)_{d-r} = 0.
\end{equation*}
(b)(3):The first part of this claim follows immediately by (b)(6) and the exact sequence used in the proof of (b)(2). Now, the second part of (b)(3) comes immediately from (a)(3).\\
(b)(1): The required vanishing conditions
$$h^i (\mathbb{P}^r , \mathcal{I}_Y (d-r+3-i+k))=0 \mbox{ for } i=1,2,3 \mbox{ and all } k \geq 0$$
are obtained respectively by (a)(1), (b)(2), (b)(3) and (b)(6).\\
Finally, the inequality for $\mathrm{e}(X)$ stated in (a)(1), follows from (b)(3) applied with $j=1.$\\
(c): By (a)(3), we know that $\depth(X) \leq 2$. Also (b)(2) implies that, if $\depth(X) =1$ then $\depth(Y) =1$ and, if $\depth(X) =2$ then $\depth(Y) =2$ or $3$. Thus we need only to show (c)(1) and (c)(3). If $d \leq 2r-4= 2(r-1) -2$, then $\mathcal{C}_h \cup \mathbb{L}_h$ is arithmetically Cohen-Macaulay by \cite[Proposition 3.5]{BS2}. Thus we have $\depth(Y) =3$ and hence $\depth(X) =2$. On the other hand, if $d \geq 3r-6$, then $\depth (\mathcal{C}_h \cup \mathbb{L}_h)=1$ by Proposition~\ref{proposition 2.3} and hence $\depth(Y) \leq 2$. Therefore either $\tau(X)=(1,1)$ or $\tau (X) = (2,2)$.\\
(d): Since $X$ is linearly normal we have $h^0 (\mathbb{P}^r , \mathcal{I}_X (2)) = h^0 (\mathbb{P}^{r-1} , \mathcal{I}_{\mathcal{C}_h} (2))$.
Moreover, by \cite[Proposition 3.6]{BS2}, we have
\begin{equation*}
h^0 (\mathbb{P}^{r-1} , \mathcal{I}_{\mathcal{C}_h} (2)) \geq {{r}\choose{2}}-d-1
\end{equation*}
where equality holds if and only if $\mathcal{C}_h \cup \mathbb{L}_h$ is arithmetically Cohen-Macaulay, or equivalently, if and only if $\depth(Y) = 3$. Finally, we know by $(c)$ that $\depth(Y) = 3$ if and only if $\tau(X)=(2,3)$. This completes the proof. \qed

\subsection*{Simplicity of the socle of the second cohomology} As a first application and extension of Theorem~\ref{4.14'' Theorem} we show that (in the previous notation)
the vanishing condition $h^2(\mathbb{P}^r,\mathcal{I}_Y) = 0$ which occurs in statement (b)(7) of that Theorem is equivalent to the simplicity of the socle of the second total
cohomology module $H^2_{*}(\mathbb{P}^r,\mathcal{I}_X) = \bigoplus_{j\in \mathbb{Z}}H^2(\mathbb{P}^r, \mathcal{I}_X(j))$ of $\mathcal{I}_X$. To formulate our result, we recall the following notation.

\begin{notation and reminder}\label{4.16''' Notation and Reminder} Let $T = \bigoplus _{n \in \mathbb{Z}} T_n$ be a graded $S$-module. Then, we denote the
\textit{socle} of $T$ by ${\rm Soc}(T)$, thus:
$${\rm Soc}(T) := (0:_T S_+) \cong \Hom_S(\Bbbk,T) = \Hom_S(S/S_+,T).$$
Keep in mind that the socle of a graded Artinian $S$-module $T$ is a
$\Bbbk$-vector space of finite dimension which vanishes if and only
if $T$ does.
\end{notation and reminder}

\begin{proposition}
\label{proposition:soc.eq}   Let the notations and hypotheses be as in Convention and Notation~\ref{convention and notation}. Then following statements are equivalent:

\begin{itemize}
\item[\rm{(i)}] $\mathrm{e}(X)$ takes its minimally possible value $\binom{d-r+2}{2}.$

\item[\rm{(ii)}] $h^2(\mathbb{P}^r,\mathcal{I}_Y) = 0.$

\item[\rm{(iii)}] $h^2(\mathbb{P}^r,\mathcal{I}_Y(j)) = 0$ for all $j \in \{0,1,2,\ldots,d-r-1\}.$

\item[\rm{(iv)}] $h^2(\mathbb{P}^r,\mathcal{I}_Y(j)) = 0$ for all $j \in \mathbb{Z}.$

\item[\rm{(v)}] $h^2(\mathbb{P}^r,\mathcal{I}_X(j)) = \binom{-j+d-r+2}{2}$ for all
$j \in \{0,1,2,\ldots,d-r\}.$

\item[\rm{(vi)}] $h^2(\mathbb{P}^r,\mathcal{I}_X(j)) = \mathrm{max}\{0, \binom{-j+d-r+2}{2}\}$ for all $j \geq 0.$

\item[\rm{(vii)}] $\mathrm{dim}_{\Bbbk}\big(\mathrm{Soc}(H^2_{*}(\mathbb{P}^r,\mathcal{I}_X))\big) = 1.$
\end{itemize}
\end{proposition}

\begin{proof} Indeed, the equivalences (i) $\Leftrightarrow$ (ii) $\Leftrightarrow$ (iii) $\Leftrightarrow$ (iv) $\Leftrightarrow$ (v) $\Leftrightarrow$ (vi) follow by
Theorem~\ref{4.14'' Theorem}. It remains to show the equivalence (iv) $\Leftrightarrow$ (vii). Consider the exact sequence of graded $S$-modules
$$0 \longrightarrow H^2_{*}(\mathbb{P}^r, \mathcal{I}_Y) \longrightarrow H^2_{*}(\mathbb{P}^r, \mathcal{I}_Y) \longrightarrow H^3_{*}(\mathbb{P}^r,\mathcal{I}_{\mathbb{F}}(-d+r-3)).$$
As
\begin{align*}
\mathrm{Soc}\big(H^3_{*}(\mathbb{P}^r,\mathcal{I}_{\mathbb{F}}(-d+r-3))\big) &= \Bbbk(d-r),\\
h^2(\mathbb{P}^r, \mathcal{I}_Y(j)) &= 0 \mbox{ for all } j \geq d-r, \mbox{ and }\\
h^2(\mathbb{P}^r, \mathcal{I}_X(j)) &= 0 \mbox{ for all } j \geq d-r+1
\end{align*}
we get an isomorphism of graded $S$-modules
$$\mathrm{Soc}\big(H^2_{*}(\mathbb{P}^r, \mathcal{I}_Y)\big) \cong \mathrm{Soc}\big(H^2_{*}(\mathbb{P}^r, \mathcal{I}_X)\big)_{\leq d-r-1}.$$
From this isomorphism, we see that
$$H^2_{*}(\mathbb{P}^r, \mathcal{I}_Y) = 0 \mbox{ if and only if } \mathrm{Soc}\big(H^2_{*}(\mathbb{P}^r, \mathcal{I}_X)\big)_{\leq d-r-1} = 0.$$
By Theorem~\ref{4.14'' Theorem} (b)(7) the module $H^2_{*}(\mathbb{P}^r, \mathcal{I}_Y)$ vanishes if and only if the number $h^2(\mathbb{P}^r, \mathcal{I}_Y)$ does.
So, condition (iv)  holds, if and only if $\mathrm{Soc}\big(H^2_{*}(\mathbb{P}^r, \mathcal{I}_X)\big)$ is concentrated in degrees $\geq d-r$. By Theorem~\ref{4.14'' Theorem} (a)(4) this is the case if and only if condition (vii) holds.
\end{proof}

\begin{remark}
\label{remark.eq.cond} (A) If the above equivalent conditions (i) -- (vii) hold, we must have
$$\tau(X) \in \{(1,1), (2,3)\}.$$
Moreover $\tau(X) = (2,3)$ implies the above equivalent conditions (i) -- (vii).\\
(B) Observe, that the above minimality condition (i) describes a \textit{generic situation}. So, it is noteworthy that the simplicity of the socle of $H^2_{*}(\mathbb{P}^r,\mathcal{I}_X)$ occurs in the generic situation, too. Below, we shall see, that in such a generic situation, a number of additional conclusions may be drawn.
\end{remark}

\begin{corollary}\label{corollary:coh.gen.}
Let the notations and hypotheses be as in Convention and Notation~\ref{convention and notation}. Then, the following statements hold:
\begin{itemize}
\item[\rm{(a)}] $h^1(\mathbb{P}^r, \mathcal{I}_X(2)) \leq h^0(\mathbb{P}^r, \mathcal{I}_X(2))-\binom{r}{2} +d+1.$
\item[\rm{(b)}] If the equivalent conditions (i)-(vii) of Proposition \ref{proposition:soc.eq} hold, the $S$-module $H^1_{*}(\mathbb{P}^r, \mathcal{I}_X)$ is minimally generated by $h^0(\mathbb{P}^r, \mathcal{I}_X(2))-\binom{r}{2} +d+1$ homogeneous elements of degree $2$.
\end{itemize}
\end{corollary}
\begin{proof}
In view of Theorem \ref{4.14'' Theorem}(b)(2) we may replace $X$ by $Y$. We choose $h \in \mathbb{U}(X)$ and apply cohomolgy to the induced exact sequence of sheaves
\begin{equation*}
0 \rightarrow \mathcal{I}_Y (-1) \stackrel{h}{ \rightarrow} \mathcal{I}_Y \rightarrow \mathcal{I}_{\mathcal{C}_h \cup \mathbb{L}_h } \rightarrow 0
\end{equation*}
in order to end up with an exact secuence of graded $S$-modules
\begin{equation*}
0 \rightarrow H^1_{*}(\mathbb{P}^r, \mathcal{I}_Y)/hH^1_{*}(\mathbb{P}^r, \mathcal{I}_Y)\rightarrow H^1_{*}(\mathbb{H}_h, \mathcal{I}_{\mathcal{C}_h \cup \mathbb{L}_h }) \rightarrow H^2_{*}(\mathbb{P}^r, \mathcal{I}_Y)
\end{equation*}
(a): By Proposition 3.6 of \cite{BS2}, the $S$-module $H^1_{*}(\mathbb{H}_h, \mathcal{I}_{\mathcal{C}_h \cup \mathbb{L}_h })$ is minimally generated by $$h^0(\mathbb{H}_h, \mathcal{I}_{\mathcal{C}_h}(2))-\binom{r}{2} +d+1$$ homogeneous elements of degree $2$. As $X$ is linearly normal, we have $ h^0(\mathbb{H}_h, \mathcal{I}_{\mathcal{C}_h}(2)) = h^0(\mathbb{P}^r, \mathcal{I}_X(2))$. Now, our claim follows immediately.\\
(b): By our hypothesis, the third module in the above sequence vanishes. Now, we get our claim by Nakayama.
\end{proof}

\begin{corollary}\label{corollary:Hilb.funct.}
Let the notations and hypotheses be as in Convention and Notation~\ref{convention and notation}. Then, the following statements hold:
\begin{itemize}
\item[\rm{(a)}] For all $j \in \mathbb{N}_0$ we have
$$h^0(X,\mathcal{O}_X(j)) = d\binom{j+1}{2}+j+1 +h^2(\mathbb{P}^r, \mathcal{I}_X(j))- \e(X).$$
\item[\rm{(b)}] If the equivalent conditions (i)-(vii) of Proposition \ref{proposition:soc.eq} hold, then
$$h^0(X,\mathcal{O}_X(j)) = \begin{cases}d\binom{j+1}{2}+j+1 + \binom{d-r+2-j}{2}-\binom{d-r+2}{2} &\mbox{ for } 0 \leq j \leq d-r, \\ \\
d\binom{j+1}{2}+j+1 -\binom{d-r+2}{2} &\mbox{ for } d-r < j.\end{cases}$$
\end{itemize}
\end{corollary}
\begin{proof} (a): Once more, let $B$ be the homogeneous coordinate ring of the projecting scroll $\widetilde{X} \subset \mathbb{P}^{d+1}$ of $X$, let $D := D_{A_+}(A) = \Gamma_{*}(X,\mathcal{O}_X)$ be the $A_+$-transform of $A$ and consider the short exaxt sequence $0 \rightarrow D \rightarrow B \rightarrow C \rightarrow 0$ in which $C$ is a one-dimensional Cohen-Macaulay module of dimension $1$ with
$$\dim_\Bbbk(C_j) = \e(X) - h^2(\mathbb{P}^r, \mathcal{I}_X(j)) \mbox{ for all } j \in \mathbb{Z}.$$ As
$$\dim_\Bbbk(B_j)= \chi(\mathcal{O}_{\widetilde{X}}(j)) = d\binom{j+1}{2}+j+1 \mbox{ for all } j \in \mathbb{N}_0,$$
we get our claim.\\
(b): This follows immediately from statement (a) bearing in mind the values of $\e(X)$ and of $h^2(\mathbb{P}^r, \mathcal{I}_X(j))$ imposed by the conditions (i) and (vi) of Proposition \ref{proposition:soc.eq}.
\end{proof}

\subsection*{The second deficiency module} We first remind the notion of deficiency module.

\begin{reminder}
\label{reminder.def.mod} Let $A$ be the homogeneous coordinate ring of the surface $X \subset \mathbb{P}^r$, let $M$ be a finitely generated graded $A$-module and let $i \in \mathbb{N}_0$. Then, the $i$-th \textit{deficiency module} of $M$ is defined by
$$K^i(M) := \mathrm{Ext}^{r-i+1}_S(M,S(-r-1)) \cong H^i(M)^{\vee},$$
where $\bullet ^{\vee} := {}^*\mathrm{Hom}_{\Bbbk}(\bullet, \Bbbk)$ denotes the (contravariant exact) graded Matlis duality functor. The module $K^{\dim(M)}(M)$ is called the
\textit{canonical module} of $M$.
\end{reminder}
In this subsection, we are interested in the second deficiency module
$$K^2(A) = \mathrm{Ext}^{r-1}_S(A,S(-r-1)) = H^2(A)^{\vee}$$ of the coordinate ring $A$ of $X$ and its induced sheaf
$$\mathcal{K}^2_X := \widetilde{K^2(A)}.$$

\begin{proposition}
\label{proposition:def.mod} Let the notations and hypotheses be as in Convention and Notation~\ref{convention and notation}. Let
$\pi = \pi_\Lambda:\widetilde{X} \twoheadrightarrow X$ be the standard normalization of $X$. Let $B$ be the homogeneous coordinate ring of the projecting scroll $\widetilde{X} \subset \mathbb{P}^{d+1}$, let $D := D_{A_+}(A) = \Gamma_{*}(X,\mathcal{O}_X) (\subseteq B)$ and set $\mathcal{F} := \pi_{*}(\mathcal{O}_{\widetilde{X}})/\mathcal{O}_X = \widetilde{B/D}.$ Then, the following statements hold:
\begin{itemize}
\item[\rm{(a)}] $K^2(A) \cong K^1(B/D)$ is a Cohen-Macaulay module of dimension one, and
                \begin{itemize}
                \item[\rm{(1)}] $\mathrm{reg}(K^2(A)) = 0;$
                \item[\rm{(2)}] $\mathcal{K}^2_X \cong \widetilde{\Gamma_{*}(X,\mathcal{F})^{\vee}}.$
                \end{itemize}
\item[\rm{(b)}] For each closed point $x \in X$, the stalk $\mathcal{K}^2_{X,x}$ of $\mathcal{K}^2_X$ at $x$ coincides with the first deficiency module
               $K^1(\mathcal{O}_{X,x}) = \mathrm{Ext}^{r-1}_{\mathcal{O}_{X,x}}(\mathcal{O}_{X,x},\mathcal{O}_{\mathbb{P}^r,x})$ of the local ring $\mathcal{O}_{X,x}$ of $X$ at $x$.
               In particular
               \begin{itemize}
               \item[\rm{(1)}] $\mathrm{length}_{\mathcal{O}_{X,x}}(\mathcal{K}^2_{X,x}) = \mathrm{e}_x(X)$ for all closed points $x \in X;$
               \item[\rm{(2)}] $\mathrm{Supp}(\mathcal{K}^2_X) = \mathrm{Sing}(X);$
               \item[\rm{(3)}] $\mathrm{length}_{\mathcal{O}_X}(\mathcal{K}^2_X) = \mathrm{e}(X).$
               \end{itemize}
\item[\rm{(c)}] If the equivalent conditions of Proposition~\ref{proposition:soc.eq} hold, then $K^2(A) \cong S/J(d-r)$, where $J \subset S$ is a saturated graded ideal such
                that:
               \begin{itemize}
               \item[\rm{(1)}] $I+L \subset J;$
               \item[\rm{(2)}] $\reg(J) = d-r+1;$
               \item[\rm{(3)}] $J$ is minimally generated by $L$ and $d-r+2$ forms of degree $d-r+1.$
               \end{itemize}
\item[\rm{(d)}] If the equivalent conditions of Proposition~\ref{proposition:soc.eq} do not hold, then the $S$-module $K^2(A)$ is minimally generated by one element of degree
                $r-d$ and some additional elements of degrees $> r-d.$

\end{itemize}
\end{proposition}
\begin{proof} (a):  Let $C := B/D.$  As seen in the proof of Theorem~\ref{theorem 2.9} (d) we have
$$H^2(A) \cong H^2(D) \cong H^1(C), \quad \widetilde{D_{A_+}(C)} \cong \widetilde{C} \cong \mathcal{F} \mbox{ and } D_{A_+}(C) \cong \Gamma_{*}(X,\mathcal{F}).$$
Applying the functor $\bullet^\vee$ to the first isomorphism, we get $K^2(A) \cong K^1(C) = K^1(B/A)$. As $C$ is a Cohen-Macaulay module of dimension $1$, so is its canonical module $K^1(C)$. \\
To prove claim (1), keep in mind that $K^2(A)$ is a Cohen-Macaulay module of dimension $1$, so that indeed
\begin{align*}
\reg(K^2(A)) &= \mathrm{end}\big(H_1(K^2(A))\big) +1 = \sup\{j \in \mathbb{Z} \mid \mathrm{dim}_{\Bbbk}(K^2(A)_{j-1}) < \mathrm{dim}_{\Bbbk}(K^2(A)_j )\}+1 \\
             & = \sup\{j \in \mathbb{Z} \mid \mathrm{dim}_{\Bbbk}(H^2(A)_{-j+1}) < \mathrm{dim}_{\Bbbk}(H^2(A)_{-j})\}+1 = -1+1 = 0.
\end{align*}
To prove claim (2), we apply the Matlis duality functor
$\bullet^\vee$ to the short exact sequence $0\rightarrow C
\rightarrow D_{A_+}(C) \rightarrow H^2(A) \rightarrow 0$ and get an
exact sequence of graded $A$-modules $0 \longrightarrow K^2(A)
\longrightarrow D_{A_+}(C)^\vee \longrightarrow C^\vee
\longrightarrow 0$.
As $\widetilde{C^\vee} = 0$,  we obtain $\mathcal{K}^2_X \cong \widetilde{\Gamma_{*}(X,\mathcal{F})^{\vee}}.$  \\
(b): The standard isomorphism $\widetilde{\mathrm{Ext}^{r-1}_S(A,S(-r-1))}_x \cong \mathrm{Ext}^{r-1}_{\mathcal{O}_{X,x}}(\mathcal{O}_{X,x},\mathcal{O}_{\mathbb{P}^r,x})$ implies the requested isomorphism $\mathcal{K}^2_{X,x} \cong K^1(\mathcal{O}_{X,x}).$ As taking local Matlis duals preserve lengths, local duality implies that
$$\mathrm{e}_x(X) = \mathrm{length}\big(H^1_{\mathfrak{m}_{X,x}}(\mathcal{O}_{X,x})\big) = \mathrm{length}\big(K^1(\mathcal{O}_{X,x})\big),$$
and this implies claim (1). Now, claims (2) and (3) are immediate.\\
(c): As $H^2(A) \cong H^2_{*}(\mathbb{P}^r,\mathcal{I}_X )$ our
hypotheses imply that $\mathrm{Soc}(H^2(A)) = H^2(A)_{d-r}\cong
\Bbbk,$ and hence the standard isomorphism
$\mathrm{Soc}\big(H^2(A)\big)^\vee \cong H^2(A)^\vee \otimes_A
\Bbbk$ shows that $K^2(A)$ is generated by one single element of
degree $r-d$. Therefore $K^2(A) \cong S/J(d-r)$ for some homogeneous
ideal $J \subset S$. By statement (a), this ideal $J$ is saturated.
If we apply cohomology to the exact sequence $0\rightarrow S/I\cap L
\rightarrow S/I \oplus S/L \rightarrow S/(I+L) \rightarrow 0$ and
keep in mind, that
$H^2(S/I\cap L) = H^2_{*}(\mathbb{P}^r,\mathcal{I}_{X \cup \mathbb{F}(X)}) = 0$, we get a monomorphism $H^2(A) \rightarrow H^2(S/(I+L))$.
Therefore $(I+L)K^2(A) = 0$ and hence $I+L \subseteq J$. As $K^2(A)$ is of dimension $1$, the inclusion is strict. This proves claim (1).\\
According to Statement (a)(1) we have $\reg(J) = \reg(S/J)+1 = \reg(K^2(A)(r-d))+1 = \reg(K^2(A)) + d-r+1 = d-r+1$, and this proves claim (2).\\
Moreover, in view of condition (v) of Proposition~\ref{proposition:soc.eq} we have
$$\dim_\Bbbk(S/J)_k = \dim_\Bbbk\big(H^2(A)_{r-d+k}\big) = \binom{\min\{k,d-r\}+2}{2} \mbox{ for all } k \in \mathbb{N}_0,$$
so that
$$\dim_\Bbbk\big((S/J)_k\big) = \dim_\Bbbk\big((S/L)_k\big) \mbox{ for all } k \leq d-r \mbox{ and }$$
$$\dim_\Bbbk\big((S/J)_{d-r+1} \big) = \dim_\Bbbk\big((S/L)_{d-r+1}\big) - (d-r+2).$$
In view of claims (1) and (2) this proves claim (3).\\
(d): If the equivalent conditions of Proposition~\ref{proposition:soc.eq} do not hold, then $\dim_\Bbbk\big(\mathrm{Soc}(H^2(A))_j\big)$ vanishes for all $j > d-r$, takes the value $1$ for $j=d-r$ and does not vanish for some $j < d-r$ (see Theorem~\ref{4.14'' Theorem} (a)(3)). Now, we get our claim by the isomorphism $\mathrm{Soc}\big(H^2(A)\big)^\vee \cong H^2(A)^\vee \otimes_A \Bbbk.$
\end{proof}

\begin{remark}\label{Remark:def.mod} Observe that the previous proposition generalizes Theorem 3.6 (e) of
\cite{BS6}.
\end{remark}

\subsection*{An application in higher dimensions} We now draw a conclusion for higher dimensional varieties of maximal sectional regularity.

\begin{corollary}\label{cor:higherdim}
Let $n \geq 2$ and let $X \subset \mathbb{P}^r$ be an $n$-dimensional variety of maximal sectional regularity of degree $d$ and of type II. Then $X$ is linearly normal and we have
\begin{equation*}
h^0 (\mathbb{P}^r ,\mathcal{I}_X (2)) \geq {{r-n+1}\choose{2}}-d-1
\end{equation*}
with equality if and only if $X \cup \mathbb{F}(X)$ is arithmetically Cohen-Macaulay. Moreover, if equality is attained, then $\depth(X)=n$.
\end{corollary}

\begin{proof}
For $n = 2,$ our claim follows by Theorem~\ref{4.14'' Theorem} (a)(2) and (d).\\
So let $n > 2$. Note that a general hyperplane section
$X' = X \cap \mathbb{H} \subset \mathbb{H}, (\mathbb{H} =\mathbb{P}^{r-1})$ of $X$ is again a variety of maximal sectional regularity and of degree $d$ of type II. So, by induction and on use of the exact sequence
$$0 \longrightarrow \mathcal{I}_X (-1) \longrightarrow  \mathcal{I}_X \longrightarrow  \mathcal{I}_{X'}  \longrightarrow 0$$
it first follows that $h^1(\mathbb{P}^r,\mathcal{I}_X (1)) = 0$ and then that $h^0 (\mathbb{P}^r ,\mathcal{I}_X (2)) = h^0 (\mathbb{H} ,\mathcal{I}_{X'}(2)) \geq {{r-n+1}\choose{2}}-d-1.$ So, $X$ is linearly normal and satisfies the requested inequality.\\
As
$\mathbb{F}(X') = \mathbb{F}(X \cap \mathbb{H})= \mathbb{F}(X) \cap \mathbb{H}$
(see \cite[Lemma 5.1 (a)]{BLPS1}), we have $\big(X \cup \mathbb{F}(X)\big)\cap \mathbb{H} = X' \cup \mathbb{F}(X')$. So, $\big(X \cup \mathbb{F}(X)\big)\cap \mathbb{H}$ is arithmetically Cohen-Macaulay if and only $X' \cup \mathbb{F}(X')$ is. Therefore, again by induction, equality holds if and only if $X \cup \mathbb{F}(X)$ is arithmetically Cohen-Macaulay.\\
Finally, by~\cite[Theorem 7.1]{BLPS1} we know that $X \cap \mathbb{F} \subset \mathbb{F}$ is a hypersurface. Therefore $\depth(X) = n$ if $X \cup \mathbb{F}(X)$ is arithmetically Cohen-Macaulay.
\end{proof}

\subsection*{Comparing Betti numbers} We finish this section with a comparison of the Betti numbers of $X$ and $Y = X \cup \mathbb{F}(X).$

\begin{proposition}\label{prop:BettiNumbers}
Let the notations and hypotheses be as in Convention and Notation~\ref{convention and notation} and assume that $X$ is not a cone. Set $m:= \reg(Y)$. Then the following statements hold:
\begin{itemize}
\item[\rm{(a)}] For all $i \geq 1$ we have
$$\beta_{i,j}(X) = \begin{cases}
                              \beta_{i,j}(Y)                           & \mbox{for $1 \leq j \leq m-1$},\\
                              \beta_{i,j}(Y) = 0                       & \mbox{for  $m \leq j \leq d-r+1$},\\
                              \beta_{i,d-r+2}(Y) + \binom{r-2}{i-1} & \mbox{for $j=d-r+2$}.
\end{cases}$$
\item[\rm{(b)}] $m \leq d-r+2$ if and only if $\beta_{i,d-r+2}(X) = \binom{r-2}{i-1}$ for all $i \geq 1$.
\end{itemize}
\end{proposition}

\begin{proof}
Let $I$ and $L$ respectively denote the homogeneous vanishing ideals of $X$ and $\mathbb{F}(X)$ in $S = \Bbbk[x_0,x_1,\ldots,x_r]$, so that $\beta_{i,j}(X) = \beta_{i,j}(S/I)$ and $\beta_{i,j}(Y) = \beta_{i,j}(S/I\cap L)$ for all $i,j \in \mathbb{N}$.\\
(a): The exact sequence
\begin{equation*}
0 \rightarrow (S/L)(-d+r-3) \rightarrow S/I\cap L \rightarrow S/I \rightarrow 0
\end{equation*}
used in the proof of Theorem \ref{4.14'' Theorem}  induces the following long exact sequence:
$$\Tor^S_i(\Bbbk,S/L)_{i+j-d+r-3} \rightarrow \Tor^S_i(\Bbbk,S/I \cap L)_{i+j} \rightarrow \Tor^S_i(\Bbbk,S/I)_{i+j}$$
$$\rightarrow\Tor^S_{i-1}(\Bbbk,S/L)_{(i-1)+j-d+r-2} \rightarrow \Tor^S_{i-1}(\Bbbk,S/I\cap L)_{(i-1)+j+1}$$
For all non-negative integers $k$ we have
$$\dim_K\big(\Tor^S_k(\Bbbk,S/L)_{k+l}\big) = \beta_{k,l}(S/L) = \begin{cases}
                                                     0 & \mbox{if $l \neq 0$}\\
                                                     \binom{r-2}{k} & \mbox{if $l = 0$}
\end{cases}$$
Therefore, the above long exact sequence makes us end up with isomorphisms
$$\Tor^S_i(\Bbbk,S/I \cap L)_{i+j} \cong \Tor^S_i(\Bbbk,S/I)_{i+j} \mbox{ for all } i \geq 1 \mbox{ and all } j \in \{1,2,\ldots, d-r+1\}.$$
As $\reg(S/I\cap L) = \reg(Y) - 1 = m-1$, we have $\beta_{i,j}(S/I\cap L) = 0 \mbox{ for all } i \geq 1 \mbox{ and all } j \geq m$.
So by the above isomorphisms we get the requested values of $\beta_{i,j}(S/I)$ for all $i \geq 1$ and all $j \in\{1,\ldots,d-r+1\}$.

As $\reg(S/I\cap L) = \reg(Y) -1 \leq d-r+2$ (see Theorem \ref{4.14'' Theorem}(b)(1)), the last module in the above exact sequences vanishes for $j=d-r+2$. So, our previous observation on the Betti numbers $\beta_{k,l}(S/L)$ yields a short exact sequence
\begin{equation*}
0\rightarrow \Tor^S_i(\Bbbk,S/I\cap L)_{i+d-r+2} \rightarrow \Tor^S_i(\Bbbk,S/I)_{i+d-r+2} \rightarrow \Bbbk^{\binom{r-2}{i-1}} \rightarrow 0 \mbox{ for all } i \geq 1,
\end{equation*}
which shows that $\beta_{i,d-r+2}(S/I) = \beta_{i,d-r+2}(S/I \cap L) + \binom{r-2}{i-1}$, and this proves our claim.\\
(b): As already said above, we have $m \leq d-r+3$ and hence $\beta_{i,j}(Y) = 0$ for all $i \geq 1$ and all $j \geq d-r+3$. Thus $m \leq d-r+2$ if and only if $\beta_{i,d-r+2} (Y)=0$ for all $i \geq 1$. Also by statement (a), the second condition holds if and only if $\beta_{i,d-r+2}(X) = \binom{r-2}{i-1}$ for all $i \geq 1$.
\end{proof}

\subsection*{Special extremal secant lines} In case of surfaces of type II, the special extremal secant locus is easily understood.
We also shall see that proper $3$-secant lines which meet $X$ only
in regular points are already special extremal lines and we shall
approximate the singular locus of $X$ by the singular locus of the
intersection of $X$ with the extremal $\mathbb{F}(X)$ plane of $X$.

\begin{proposition}
\label{prop:extseclinesII} Let the hypotheses be as in Convention
and Notation~\ref{convention and notation}. Then the following
statements hold:
\begin{itemize}
 \item[\rm{(a)}] The  image $\psi\big({}^*\Sigma(X)\big)$ of the special extremal locus ${}^*\Sigma(X)$ of $X$ under the Pl\"ucker embedding
$$\psi:\mathbb{G}(1,\mathbb{P}^r\big) \longrightarrow \mathbb{P}^{\binom{r+1}{2}-1}$$
is a plane.
\item[\rm{(b)}] Let $\mathbb{L} \in \mathbb{G}(1,\mathbb{P}^r)$ such that $\mathbb{L} \nsubseteq X.$ Then, the following statements are equivalent:
\begin{itemize}
               \item[\rm{(i)}]  $\mathbb{L} \in {}^*\Sigma(X);$
               \item[\rm{(ii)}] $\mathrm{length}(X \cap \mathbb{L}) \geq 3$ and $X \cap \mathbb{L} \subset \mathrm{Reg}(X).$
\end{itemize}
\item[{(c)}] $\mathrm{Sing}(X) = \{x \in \mathrm{Sing}\big(X \cap \mathbb{F}(X)\big) \mid \mathbb{F}(X) \subsetneq \mathrm{T}_x(X)\}$. In particular, each point $x \in \mathrm{Sing}\big(X \cap \mathbb{F}(X)\big)$ not contained in a line $\mathbb{L} \subset X \cap \mathbb{F}(X)$ is a singular point of $X$.
\end{itemize}
\end{proposition}

\begin{proof} (a): As $X$ is of type II, we know that $\overline{\bigcup_{\mathbb{L} \in {}^*\Sigma(X)} \mathbb{L}} = \mathbb{F}(X) = \mathbb{P}^2 \subset \mathbb{P}^r$ is a plane, so that
${}^*\Sigma(X) = \mathbb{G}(1,\mathbb{F}(X)) = \mathbb{G}(1,\mathbb{P}^2)$. Standard arguments on Pl\"ucker embeddings show that $\psi\big(\mathbb{G}(1,\mathbb{P}^2)\big)$ is a plane in $\mathbb{P}^{\binom{r+1}{2}-1}$. \\
(b): The implication ``(i) $\Leftarrow $ (ii)" follows as
$\mathcal{C}_\mathbb{H} := X \cap \mathbb{H}$ is smooth for each
$\mathbb{H} \in \mathcal{U}(X)$ and hence can only contain
smooth points of $X$. \\
So, let $\mathbb{L} \in \Sigma_3(X)$ such that $X \cap \mathbb{L}$
is finite and contained in $\Reg(X)$, and assume that
$\mathbb{L}_\mathbb{H} \neq \mathbb{L}$ for all $\mathbb{H} \in
\mathcal{U}(X)$. We aim for a contradiction. Let $\pi =
\pi_\Lambda:\widetilde{X} \twoheadrightarrow X$ be the standard
normalization of $X$, induced by the linear projection $\pi' =
\pi'_\Lambda : \mathbb{P}^{d+1} \twoheadrightarrow \mathbb{P}^r$,
consider the closed preimage $\mathbb{L}' :=
\overline{\pi'^{-1}(\mathbb{L})} \in
\mathbb{G}(d-r+2,\mathbb{P}^{d+1}),$ of $\mathbb{L}$ and observe
that $\widetilde{X} \cap \mathbb{L}' = \pi^{-1}(X\cap\mathbb{L})$ is
finite. Let $\mathbb{H}' \in \mathbb{G}(d,\mathbb{P}^{d+1})$ be a
general hyperplane which contains the space $\mathbb{L}'$. If
$\widetilde{X}$ is not a cone, we may conclude by \cite[Remark 2.3
(B)]{BP2}, that the intersection $\widetilde{X} \cap \mathbb{H}'
\subset \mathbb{H}'$ is a rational normal curve. If $\widetilde{X}$
is a cone, the fact that $\mathbb{L}$ avoids the singular locus of
$X$ implies that $\mathbb{L}'$ does not contain the vertex of
$\widetilde{X}$ and we end up again with the conclusion that
$\widetilde{X} \cap \mathbb{H}' \subset \mathbb{H}'$ is a rational
normal curve. As $\mathbb{H}'$ is general, the hyperplane
$\mathbb{H} := \pi'(\mathbb{H}' \setminus \Lambda)
\in \mathbb{G}(r-1,\mathbb{P}^r)$ avoids the finite set $\mathrm{Sing}(\pi)$, and hence $\mathcal{C}_\mathbb{H} = X \cap \mathbb{H} = \pi(\widetilde{X}\cap\mathbb{H}') \subset \mathbb{P}^r$ is an integral curve, whence $\mathbb{H} \in \mathcal{U}(X).$\\
By our assumption we have $\mathbb{L} \neq \mathbb{L}_\mathbb{H}$, hence $\mathbb{V} := \langle\mathbb{L}_\mathbb{H},\mathbb{L}\rangle \in \mathbb{G}(s,\mathbb{H})$
with $s \in \{2,3\}$. In particular, the intersection $X \cap \mathbb{V}$ is finite. As $\mathbb{L} \neq \mathbb{L}_\mathbb{H}$, we have
$$\mathrm{length}(X \cap (\mathbb{L} \cup \mathbb{L}_\mathbb{H})) \geq \mathrm{length}(X \cap \mathbb{L}) + \mathrm{\length}(X \cap \mathbb{L}_\mathbb{H}) - \varepsilon$$
with $\varepsilon = 1$ if $\mathbb{L} \cap \mathbb{L}_\mathbb{H} \subset X$ and $\varepsilon = 0$ otherwise. In the first case, we have $s = 2$, so that always
$3-\varepsilon \geq s$. Therefore, we obtain
$$ \mathrm{\length}(X \cap (\mathbb{L} \cup \mathbb{L}_\mathbb{H})) \geq    \mathrm{length} (X \cap \mathbb{L}) +
\mathrm{length}(X\cap \mathbb{L}_\mathbb{H}) - \varepsilon \geq 3 + d-r+3 - \varepsilon \geq d-r+s+3.$$
As $\mathbb{L} \cup \mathbb{L}_\mathbb{H} \subset {\rm Reg}(X) \cap \mathbb{V}$ this contradicts Proposition~\ref{prop:loc,properties} (b).\\
(c): Let $x \in \mathrm{Sing}(X)$. Let $\mathbb{H} \in \mathcal{U}(X)$ such that $\mathbb{L}_{\mathbb{H}} \cap \mathrm{Sing}(X) = \emptyset$
and consider the plane $\mathbb{E}_{\mathbb{H}} := \langle x, \mathbb{L}_{\mathbb{H}} \rangle \subset \mathbb{P}^r$. Then, by Proposition~\ref{prop:loc,properties} (c), we have $\mathrm{dim}(X \cap \mathbb{E}_\mathbb{H}) = 1$.  If $x \notin \mathbb{F}(X)$, this would imply the contradiction that the intersection of $X$ with the three-space $\langle x, \mathbb{F}(X) \rangle$ contains infinitely many curves. Therefore $x \in \mathbb{F}(X)$.\\
Now, let $\mathbb{L} \subset \mathbb{F}(X)$ be a general line such
that $x \in \mathbb{L}$. Then, by Lemma~\ref{4.12'' Lemma+} (a) we
have $\mathrm{Sing}(X) \cap \mathbb{L} = \{x\}$ and $\length(X \cap
\mathbb{L}) = d-r+3$. Assume, that $\mathrm{mult}_x(\mathbb{L} \cap
X) = 1,$ so that $\mathrm{length}(\mathrm{Reg}(X) \cap \mathbb{L}) =
\mathrm{length}(X \cap \mathbb{L}) - 1 = d-r+2$. By
Proposition~\ref{prop:loc,properties} (b) it follows that $d-r+4 =
d-r+2+2 \leq d-r+1+2 = d-r+3,$ and this contradiction shows that
$\mathrm{mult}_x(\mathbb{L} \cap X) > 1$. This first shows that
$\mathbb{L}$ is a tangent line to $X$ in $x$, and hence proves that
$\mathbb{F}(X) \subseteq \mathrm{T}_x(X)$. As $x \in
\mathrm{Sing}(X)$, the inclusion is strict. As
$\mathrm{mult}_x\big(\mathbb{L} \cap (X \cap \mathbb{F}(X)\big) =
\mathrm{mult}_x(\mathbb{L} \cap X) >1$ it also follows that a
general line $\mathbb{L} \subset \mathbb{F}(X)$ which runs through
$x,$ is tangent to $X \cap \mathbb{F}(X)$ in $x$, so that $x \in
\mathrm{Sing}\big(X\cap \mathbb{F}(X)\big).$ This proves the
inclusion ''$\subseteq$'' between the two sets in question. As the
converse inclusion is obvious, we get the requested equality. The
additional claim now follows easily, as $X$ is a union of lines.
\end{proof}

\section{The index of normality of $X$}

\subsection*{Index of normality and extremal planes} Our next main result is devoted to the study of the relations among the index of normality $N(X)$, the Betti numbers $\beta_{i,j}(X)$ and the nature of the union
$X \cup \mathbb{F}(X)$, where $X$ is a surface of maximal sectional regularity which is of type II. We begin with two auxiliary results.

\begin{lemma}\label{4.16'''' Lemma} Let the notations and hypotheses be as in Convention and Notation~\ref{convention and notation}. Assume that $X$ is not a cone. Then we have the following statements
\begin{itemize}
\item[\rm{(a)}] ${\rm Soc}(H^1(S/I))(-r-1) \cong \Tor^S_r(\Bbbk,S/I)$.
\item[\rm{(b)}] If $\depth(X) = 1$, then $H^1\big(\mathbb{P}^r,\mathcal{I}_X(N(X))\big) \cong \Tor^S_r(\Bbbk,S/I)_{N(X)+r+1}.$
\item[\rm{(c)}] $N(X) \leq d-r$ if and only if $\beta_{r,d-r+2}(X) = 0$.
\end{itemize}
\end{lemma}
\begin{proof} (a): If $\depth(X) > 1$, both of the occurring modules vanish and our claim is obvious. So, we assume that $\depth(X) = 1$ and consider the total ring of
sections $D := D_{S_+}(S/I) = \bigoplus_{n \in \mathbb{Z}} H^0(\mathbb{P}^r, \mathcal{O}_X(n))$ of $X$, as well as the short exact sequence
$$0 \longrightarrow S/I \longrightarrow D \longrightarrow H^1(S/I) \longrightarrow 0.$$
We apply the Koszul functor $K(\underline{x}; \bullet)$ with respect to $\underline{x} := x_0,x_1,\ldots,x_r$ to this sequence and end up in homology with
an exact sequence
$$H_{r+1}(\underline{x};D) \rightarrow H_{r+1}(\underline{x};H^1(S/I)) \rightarrow H_r(\underline{x}; S/I) \rightarrow H_r(\underline{x}; D).$$
As $\depth(D) > 1$ the first and the last module in this sequence vanish, so that
$$H_{r+1}(\underline{x};H^1(S/I)) \cong H_r(\underline{x}; S/I).$$
As the Koszul complex $K(\underline{x},S)$ provides a free resolution of $\Bbbk = S/S_+$ and $K(\underline{x}; S/I) \cong K(\underline{x}; S) \otimes_S S/I$ we have
$H_r(\underline{x}; S/I) \cong \Tor^S_r(\Bbbk,S/I)$. As the sequence $\underline{x}$ has length $r+1$, we have $H_{r+1}(\underline{x}; H^1(S/I)) \cong
{\rm Soc}(H^1(S/I))(-r-1)$. Altogether, we now obtain the requested statement (a).\\
(b): As $N(X) = {\rm end}(H^1(S/I))$, we have
$$H^1\big(\mathbb{P}^r,\mathcal{I}_X(N(X))\big) \cong H^1(S/I)_{N(X)} = {\rm Soc}(H^1(S/I))_{N(X)}.$$
Now, our claim follows immediately by statement (a).\\
(c): If $\depth(X) > 1$ we have $N(X) = -\infty$ and $\beta_{r,d-r+2}(X) = 0$, so that our claim is true. We thus may assume that $\depth(X) = 1$.
As $\reg(X) = d-r+3$ we have $N(X) \leq d-r+1$ and $\Tor^S_r(\Bbbk,S/I)_{r + l} = 0$ for all $l \geq d-r+3$. Now, we may conclude by statement (b).
\end{proof}

Now, we are ready to give the announced main result.

\begin{theorem}\label{4.17'' Proposition} Let the notations and hypotheses be as in Convention and Notation~\ref{convention and notation}. Assume that $X$ is not a cone. Then
\begin{itemize}
\item[\rm{(a)}] The following statements are equivalent:
                \begin{itemize}
                \item[\rm{(i)}] $N(X) \leq d-r$.
                \item[\rm{(ii)}] $\reg(X \cup \mathbb{F}(X)) \leq d-r+2$.
                \item[\rm{(iii)}] $\beta_{i,d-r+2}(X) = \binom{r-2}{i-1}$ for all $i \geq 1$.
                \item[\rm{(iv)}] $\beta_{r,d-r+2}(X) = 0$.
                \end{itemize}
\item[\rm{(b)}] The following statements are equivalent:
               \begin{itemize}
               \item[\rm{(i)}]  $\beta_{1,d-r+2}(X) = 1$.
               \item[\rm{(ii)}] $I\cap L = (I_{\leq d-r+2})$, where $I$ and $L$ are the homogeneous vanishing ideals of $X$
                                respectively of $\mathbb{F}(X)$ in $S$.
               \end{itemize}
\end{itemize}
\end{theorem}
\begin{proof}
(a): (i) $\Rightarrow$ (ii): Let $N(X) \leq d-r$ and $I$ and $L
\subset S =\Bbbk[x_0,x_1,\ldots,x_r]$ respectively denote the
homogeneous vanishing ideals of $X$ and of $\mathbb{F}(X)$.
According to Theorem~\ref{4.14'' Theorem} (b)(1) we have ${\rm
end}(H^1(S/I\cap L) = {\rm end}(H^1(S/I))= N(X) \leq d-r$. So, it
follows by Theorem~\ref{4.14'' Theorem} (b)(3) and (6) that
$\reg(S/I \cap L) \leq d-r+1$, whence $\reg(X) \cup \mathbb{F}(X) =
\reg(I \cap L)
\leq d-r+2$.\\
(ii) $\Rightarrow$ (i): As ${\rm end}(H^1(S/I)) = N(X)$, this is an easy consequence of Theorem~\ref{4.14'' Theorem} (b)(2).\\
(ii) $\Rightarrow$ (iii): This follows by Propsition~\ref{prop:BettiNumbers}.\\
(iii) $\Rightarrow$ (ii): Assume that statement (iii) holds. Then we have in particular that $\beta_{1,d-r+2}(X) = 1$. Now, we may again conclude by Proposition~\ref{prop:BettiNumbers}.\\
(iii) $\Leftrightarrow$ (iv): This is clear by Lemma~\ref{4.16'''' Lemma}.\\
(b): (i) $\Rightarrow $ (ii): Assume that $\beta_{1,d-r+2}(X) = 1$. Then, it follows by Proposition~\ref{prop:BettiNumbers} (a) that $(I_{\leq d-r+2}) = I \cap L$.\\
(ii)$\Rightarrow$(i): This also follows immediately by Proposition~\ref{prop:BettiNumbers} (a).
\end{proof}

\begin{remark}\label{remark:normality} (A) The extremal plane $\mathbb{F}(X)$ of a surface $X$ of type II which satisfies $N(X) \leq d-r$ has some nice properties.
We namely can say that the equivalent properties (i),(ii) of
Theorem~\ref{4.17'' Proposition} (b) imply the following statements,
in which, for all $m \in \mathbb{N}$, we use
$$\mathrm{Sec}_m(X) := \bigcup_{\mathbb{L} \in \mathbb{G}(1,\mathbb{P}^r) : \mathrm{length}(X \cap \mathbb{L})\geq m} \mathbb{L}$$
to denote the $m$-\textit{th secant variety} of $X$.
\begin{itemize}
\item[\rm{(1)}] If $\mathbb{L} \in \Sigma(X)$, then $\mathbb{L} \subset X \cup \mathbb{F}(X),$
\item[\rm{(2)}] $\mathrm{Sec}_{d-r+3}(X) = X \cup \mathbb{F}(X),$
\item[\rm{(3)}] $\mathbb{F}^+(X) = \mathbb{F}(X).$
\end{itemize}
Indeed, assume that the equivalent statements (i) and (ii) of Theorem~\ref{4.17'' Proposition}(b) hold. Let $\mathbb{L} \in \mathbb{G}(1,\mathbb{P}^r)$
such that $d-r+3 \leq \mathrm{length}(X \cap \mathbb{L}) < \infty.$ Let $M \subset S$ be the homogeneous vanishing ideal of $\mathbb{L}$. Then, $(I_{d-r+2}) \subset M$. As $I\cap L = (I_{\leq d-r+2})$ it follows that $I \cap L \subset M$. As $\mathbb{L}$ is not contained in $X$, the ideal $I$ is not contained in $M$. It follows that $L \subset M$, and hence that $\mathbb{L} \subset \mathbb{F}(X)$. As $\Sigma(X)$ is the closure of all lines $\mathbb{L}$ as above, this proves claim (1). \\
Claims (2) and (3) are immediate by claim (1), as $\mathbb{F}(X)$ is a union of lines and each line $\mathbb{L} \subset \mathbb{P}^r$ with $\mathrm{length}(X \cap \mathbb{L}) > d-r+3$ is contained in $X$.\\
(B) Observe that statement (iii) of Theorem~\ref{4.17'' Proposition} (a) implies that $\beta_{1,d-r+2}(X) = 1$.
So, the equivalent statements (a) (i)--(iv) imply the equivalent statements (b)(i) and (b)(ii) of this theorem.\\
(C) We have seen above, that surfaces $X$ of type II  and
sub-maximal index of normality behave nicely. We therefore can
expect, that in the extremal case $N(X) = -\infty$ -- hence in the
case where $\depth(X) = 2$ -- we get even more detailed information
on the Betti numbers if $X$ is of ``small degree".
\end{remark}

\subsection*{Surfaces of degree $r+1$ in $\mathbb{P}^r$} We now briefly revisit the special case of surfaces $X \subset \mathbb{P}^r$ of degree $r+1$.

\begin{remark} \label{3.3' Remark} (s. \cite{B2}, \cite{BS6})
(A) Assume that $r \geq 5$ and let our surface $X \subset \mathbb{P}^r$ be of degree $r+1$. Then,
we can distinguish $9$ cases, which show up by their numerical invariants
as presented in the following table. Here $\sigma(X)$ denotes the \textit{sectional
genus} of $X$, that is the arithmetic genus of the generic hyperplane section curve $\mathcal{C}_h \quad
(h \in \mathbb{U}(X))$ or equivalently, the sectional genus of the polarized surface
$(X, \mathcal{O}_X(1))$ in the sense of Fujita \cite{Fu}. Moreover $\mathrm{sreg}(X)$ denotes the sectional regularity introduced in Remark and Definition~\ref{remark and definition 0}.

\[
\begin{tabular}{| c | c | c | c | c | c | c |}
    \hline
    $\rm{Case}$ & $\sreg(X)$ & $\depth(X)$
    & $\sigma(X)$ & $\e(X)$ & $h^1_A(1)$ & $h^1_A(2)$  \\ \hline
    $1$ &$2$ &$3$ &$2$ &$0$ &$0$ &$0$ \\ \hline
    $2$ &$3$ &$2$ &$1$ &$0$ &$0$ &$0$ \\ \hline
    $3$ &$3$ &$2$ &$1$ &$1$ &$0$ &$0$ \\ \hline
    $4$ &$3$ &$1$ &$1$ &$0$ &$1$ &$0$ \\ \hline
    $5$ &$3$ &$2$ &$0$ &$2$ &$0$ &$0$ \\ \hline
    $6$ &$3$ &$1$ &$0$ &$1$ &$1$ &$\leq 1$ \\ \hline
    $7$ &$3$ &$1$ &$0$ &$0$ &$2$ &$\leq 2$ \\ \hline
    $8$ &$4$ &$2$ &$0$ &$3$ &$0$ &$0$ \\ \hline
    $9$ &$4$ &$1$ &$0$ &$0$ &$2$ &$3$  \\ \hline
    \end{tabular}
\]
\\
The case $9$ occurs only if $r = 5$. In \cite{B2} and \cite{BS6} we listed indeed two more cases $10$ and $11$, of which we did not know at that time, whether they might
occur at all. For these two cases we had $\sreg(X) = 4 = d-r+3$ and ${\rm e}(X) \in \{1,2\}$. As these surfaces would be of maximal sectional regularity, this would
contradict Theorem~\ref{t1-coh} (b) and Theorem~\ref{4.14'' Theorem} (a). So, surfaces which fall under the cases $10$ and $11$ cannot occur at all. In the case $9$ we have ${\rm e}(X) = 0$, and hence by Theorem~\ref{4.14'' Theorem} (b)(3) the surface $X$ is of type I in this case. \\
(B) The surfaces of types 8 and 9 are of particular
interest, as they are the ones of maximal sectional regularity within all the 9 listed types.
Observe, that among all surfaces $X$ of degree $r+1$ in $\mathbb{P}^r$, those of type
8 are precisely the ones $X$ which are of maximal sectional regularity and of arithmetic depth
$\geq 2$. If $r \geq 6$, the surfaces of type 8 are precisely the ones which are of maximal sectional regularity.\\
(C) Observe, that in the cases 5 -- 9 we have $\sigma(X) = 0$. This
means, that the surfaces which fall under these 5 types are all
sectionally rational and have finite non-normal locus. So, by
Theorem 4.1 in \cite{BLPS1}, these surfaces are almost non-singular
projections of a rational normal surface scroll $\widetilde{X} =
S(a,r+1-a)$ with $0 \leq a \leq \frac{r+1}{2}$, even if they are
cones (see \cite[Corollary 5.11]{BS5} for the non-conic case). So,
according to Theorem~\ref{2.5' Theorem} (b) the surfaces $X$ of
types 5 -- 9 all satisfy the Eisenbud-Goto inequality $\reg(X) \leq
4$, with equality in the cases 8 and 9 (see Theorem~ \ref{4.14''
Theorem} (a)(1) and Theorem~\ref{t1-betti} ). In the cases 1 -- 5,
the values of $h^i(\mathbb{P}^r, \mathcal{I}_X(n)) =: h^i(S/I)_n
\quad (i=1,2, n\in \mathbb{Z})$ (see \cite[Reminder 2.2 (C) and
(D)]{BS6}) show, that $\reg(X) = 3$. In the case 6 we may have
$\reg(X) = 3$ whereas in the case 7, we know even that $\reg(X)$ may
take both values $3$ and $4$ (see \cite[Example 3.5, Examples 3.4
(A),(B) and (C)]{BS6}). This shows in particular, that there are
sectionally rational surfaces $X \subset \mathbb{P}^r$ of degree
$r+1$ with finite non-normal locus and $\sreg(X) < \reg(X)$.
\end{remark}

\begin{corollary} \label{3.5' Corollary}
Assume that the surface $X \subset \mathbb{P}^r_K$ is of degree $r+1$. Then, the following statements are equivalent

               \begin{itemize}
               \item[\rm{(i)}] The surface $X$ is of type 8.
               \item[\rm{(ii)}] $\e(X) = 3$.
               \item[\rm{(iii)}] $\sreg(X) = 4$ and $\depth(X) =2$.
               \item[\rm{(iv)}] $\sreg(X) = 4$ and $X$ does not fall under the case 9 of Remark~\ref{3.3' Remark}.
               \end{itemize}

\end{corollary}

\begin{proof}
This follows easily on use of the table in Remark~\ref{3.3' Remark} (A).
\end{proof}

\section{Examples and Problems}

\subsection*{Surfaces of extremal regularity with small extremal secant locus } As announced already in the Introduction, we now shall present a construction, which allows to provide examples of surfaces of extremal regularity whose extremal secant variety is of dimension $-1,0,$ or $1$. These surfaces are in particular not of maximal sectional regularity. We already have spelled out the meaning of such examples in relation what is said about varieties of extremal regularity in \cite{GruLPe}.

\begin{construction and examples}\label{7.1 Construction and Examples} Let $a,b,d \in \mathbb{N}$ with $a \leq b$, let $r:=a+b+3$, assume that $d > r$ and
consider the smooth threefold rational normal scroll of degree $a+b+1 = r-2$
\begin{equation*}
Z := S(1,a,b) \subset \mathbb{P}^r.
\end{equation*}
Let $H, F \in {\rm Div}(Z)$ respectively be a hyperplane section and a ruling plane of $Z$, so that each divisor on
$Z$ is linearly equivalent to $mH+nF$ for some integers $m,n$. Let $X \subset \mathbb{P}^r$ be an non-degenerate irreducible
surface of degree $d$ which is contained in $Z$ as a divisor linearly equivalent to $H+(d-r+2)F$.

(A) One can easily see that $h^0 (X, \mathcal{O}_X (1)) = h^0 (Z, \mathcal{O}_Z (1))+d-r+1 = d+2$. This means that the linearly normal embedding $\widetilde{X} \subset \mathbb{P}^{d+1}$ of $X$ by means of $\mathcal{O}_X (1)$ is of minimal degree and $X$ is a regular projection of $\widetilde{X}$. Keep in mind, that $\widetilde{X}$ is either a smooth rational normal surface scroll, a cone over a rational normal curve or the Veronese surface in $\mathbb{P}^5$. As $d+1 > 5$, and as a cone does not admit a proper isomorphic linear projection, $\widetilde{X} \subset \mathbb{P}^{d+1}$ is a smooth rational normal surface scroll. This means that $X$ is smooth and sectionally rational. Also $\reg (X) = d-r+3$ (cf. \cite[Theorem 4.3]{P}) and hence $X$ is a surface of extremal regularity.

(B) Let $\mathbb{L}$ be a line section of $Z$. Then the intersection number $\mathbb{L}\cdot X$ takes the maximal possible value $d-r+3$.
This means that either $\mathbb{L}$ is contained in $X$ or else it is a proper $(d-r+3)$-secant line to $X$.
On the other hand, observe that any proper $(d-r+3)$-secant line to $X$ must be contained in $Z$ as a line section, since $Z$ is cut out by quadrics. \\
To reformulate this observation, we introduce the locally closed subset
$$\Sigma^{\circ}(X) := \{\mathbb{L} \in \mathbb{G}(1,\mathbb{P}^r) \mid d-r+3 \leq \mathrm{length}(X \cap\mathbb{L}) < \infty\} \subset \mathbb{G}(1,\mathbb{P}^r)$$
of proper $(d-r+3)$-secant lines to $X$, so that
$\overline{\Sigma^{\circ}(X)} = \Sigma(X)$. The previous observation
now may be written in the form
$$\Sigma^{\circ}(X) = \{\mathbb{L} \in \mathbb{G}(1,\mathbb{P}^r) \mid \mathbb{L} \mbox{ is a line section of } Z \mbox{ and } \mathbb{L} \nsubseteq X\}.$$

(C) Suppose that $a \geq 2$ and let $\mathbb{L}$ be the unique line
section $S(1)$ of $Z$. If $\mathbb{L}$ is contained in $X$, then we
have $\Sigma^{\circ}(X) = \emptyset$ and hence $\Sigma (X) =
\emptyset$. So, in this case $X$ is a surface of maximal regularity,
having no proper extremal secant line at all. Next, if $\mathbb{L}$
is not contained in $X$, then we have $\Sigma^{\circ}(X) =
\Sigma(X)=\{ \mathbb{L} \}$, and hence $\dim\big(\Sigma(X)\big) =
0$.

(D) Suppose next, that $a =1$ and $b \geq 2$. Then, by part (B), we have
\begin{equation*}
\Sigma^{\circ}(X) = \{\mathbb{L} \in S(1,1) \mid \mathbb{L} \mbox{ is a line section of } Z \mbox{ and } \mathbb{L} \nsubseteq X\}.
\end{equation*}
Since $X \neq S(1,1)$, all but finitely many line sections of $Z$ are proper $(d-r+3)$-secant lines to $X$. This implies that $\dim\big(\Sigma(X)\big) = \dim\big(\Sigma^{\circ}(X)\big) = 1$ and that $\mathbb{F}^{+}(X)$ is exactly equal to $S(1,1)$.
\end{construction and examples}

\subsection*{Surfaces of maximal sectional regularity of type I} We now provide a few examples for surfaces of maximal sectional regularity of type I, focusing on the
various Betti tables which may occur. These tables have been computed by means of the Computer Algebra System Singular \cite{GrPf}. We use the divisorial description of surfaces of type I given in Theorem~\ref{theorem 2.1}

\begin{example}
\label{example:t1}
Let $X \subset W := S(1,1,1) \subset \mathbb{P}^{5}$ be a divisor which is linearly equivalent to $H + (d-3)F$. Then $X$ is given by an isomorphic projection of a smooth rational normal scroll $\widetilde{X} \subset \mathbb{P}^{d+1}$ (see \cite[Lemma 3.1]{P}).\\
(A) Let $d=8$ and assume that $X$ has the parametrization
\begin{equation*}
\{\big(u^7s:u^7t:vs^7:vs^6t:vst^6:vst^7\big) \mid (s,t), (u,v) \in \Bbbk^2 \setminus\{(0,0)\}\}.
\end{equation*}
Then, $X$ has the following Betti table.
\[
\begin{tabular}{|c|c|c|c|c|c|c|}\hline
                          \multicolumn{2}{|c||}{$i$}& $1$&$2$&$3$&$4$&$5$  \\\cline{1-7}
             \multicolumn{2}{|c||}{$\beta_{i,1}$}& $3$&$2$&$0$&$0$&$0$    \\\cline{1-1}
             \multicolumn{2}{|c||}{$\beta_{i,2}$}& $0$& $0$&$0$& $0$&$0$  \\\cline{1-1}
             \multicolumn{2}{|c||}{$\beta_{i,3}$}& $0$& $0$&$0$& $0$&$0$  \\\cline{1-1}
             \multicolumn{2}{|c||}{$\beta_{i,4}$}& $0$&$0$&$0$&$0$&$0$     \\\cline{1-1}
             \multicolumn{2}{|c||}{$\beta_{i,5}$}& $21$&$70$&$87$&$48$&$10$\\\cline{1-7}
             \end{tabular}
\]

(B) Let $d=9$ and assume that $X$ has the parametrization
\begin{equation*}
\{\big(u^7s:u^7t:vs^8:vs^7t:vst^7:vt^8 \big) \mid (s,t), (u,v) \in
\Bbbk^2 \setminus\{(0,0)\}\}.
\end{equation*}
Then, we get the following Betti table for $X$.
\[
\begin{tabular}{|c|c|c|c|c|c|c|}\hline
                          \multicolumn{2}{|c||}{$i$}& $1$&$2$&$3$&$4$&$5$  \\\cline{1-7}
             \multicolumn{2}{|c||}{$\beta_{i,1}$}& $3$&$2$&$0$&$0$&$0$     \\\cline{1-1}
             \multicolumn{2}{|c||}{$\beta_{i,2}$}& $0$& $0$&$0$& $0$&$0$  \\\cline{1-1}
             \multicolumn{2}{|c||}{$\beta_{i,3}$}& $0$& $0$&$0$& $0$&$0$  \\\cline{1-1}
             \multicolumn{2}{|c||}{$\beta_{i,3}$}& $0$& $0$&$0$& $0$&$0$  \\\cline{1-1}
             \multicolumn{2}{|c||}{$\beta_{i,4}$}& $0$&$0$&$0$&$0$&$0$     \\\cline{1-1}
             \multicolumn{2}{|c||}{$\beta_{i,5}$}& $28$&$96$&$123$&$70$&$15$\\\cline{1-7}
             \end{tabular}
\]

(C) Let $d=10$ and assume that $X$ has the parametrization
\begin{equation*}
\{\big(u^9s:u^9t:vs^9:vs^8t:vst^8:vst^9\big) \mid (s,t), (u,v) \in \Bbbk^2 \setminus\{(0,0)\}\}.
\end{equation*}
Then, the Betti table of $X$ is as given below.
\[
\begin{tabular}{|c|c|c|c|c|c|c|}\hline
                          \multicolumn{2}{|c||}{$i$}& $1$&$2$&$3$&$4$&$5$  \\\cline{1-7}
             \multicolumn{2}{|c||}{$\beta_{i,1}$}& $3$&$2$&$0$&$0$&$0$     \\\cline{1-1}
             \multicolumn{2}{|c||}{$\beta_{i,2}$}& $0$& $0$&$0$& $0$&$0$  \\\cline{1-1}
             \multicolumn{2}{|c||}{$\beta_{i,3}$}& $0$& $0$&$0$& $0$&$0$  \\\cline{1-1}
              \multicolumn{2}{|c||}{$\beta_{i,3}$}& $0$& $0$&$0$& $0$&$0$  \\\cline{1-1}
             \multicolumn{2}{|c||}{$\beta_{i,3}$}& $0$& $0$&$0$& $0$&$0$  \\\cline{1-1}
             \multicolumn{2}{|c||}{$\beta_{i,4}$}& $0$&$0$&$0$&$0$&$0$     \\\cline{1-1}
             \multicolumn{2}{|c||}{$\beta_{i,5}$}& $36$  &$126$&$165$  &$96$&$21$ \\\cline{1-7}
             \end{tabular}
\]
\end{example}

\subsection*{Surfaces of maximal sectional regularity of type II} Next, we aim to present examples which concern surfaces $X \subset \mathbb{P}^r$ $(r \geq 5)$ of maximal sectional regularity of degree $d > r$ and of type II. Set $Y := X \cup \mathbb{F}(X)$ and recall that $\tau(X)$ denotes the pair $(\depth (X), \depth (Y))$. We will construct a few examples of $X$, having all possible $\tau(X)$ listed in Theorem \ref{4.14'' Theorem}(c).  \\

\begin{construction and examples}\label{7.2 Construction and Examples} (A) We assume that the characteristic of the base field $\Bbbk$ is zero. Let $a,b$ be integers such that $3 \leq a \leq b$ and consider the standard smooth rational normal surface scroll $\widetilde{X} := S(a,b) \subset \mathbb{P}^{a+b+1}$. We shall construct surfaces of maximal sectional regularity of type II by projecting $\widetilde{X}$ from appropriate linear subspaces of $\mathbb{P}^{a+b+1}$. The occurring Betti diagrams have been computed by means of the Computer Algebra System Singular \cite{GrPf}.

(B) Let $\Lambda$ be an $(a-3)$-dimensional subspace of $\langle S(a) \rangle = \mathbb{P}^a$ which avoids $S(a)$ and let $X \subset \mathbb{P}^{b+3}$ be the linear projection of $\widetilde{X}$ from $\Lambda$. Observe that this linear projection maps $\langle S(a) \rangle$ onto a plane $\mathbb{P}^2 = \mathbb{F} \subset \mathbb{P}^{b+3}$. Suppose that this projection maps $S(a)$ birationally onto a plane curve $C_a \subset \mathbb{F}$ of degree $a$. Since $X \subset \mathbb{P}^{b+3}$ is a surface of degree $a+b$, we have $\reg (X) \leq a$ by Theorem \ref{4.14'' Theorem}(a). On the other hand, a general line on $\mathbb{F}$ is a proper $a$-secant line to $X$. Therefore ${\rm reg}(X)=a$, $X$ is a surface of maximal sectional regularity of type II and $\mathbb{F}(X) = \mathbb{F}$. Finally, we get $\tau(X) = (2,3)$ by Theorem \ref{4.14'' Theorem}(c).

(C) Assume that $b \geq 3$. Let $\Lambda$ be a $(b-3)$-dimensional subspace of $\langle S(b) \rangle = \mathbb{P}^b$ which avoids $S(b)$ and let $X \subset \mathbb{P}^{a+3}$ be the linear projection of $\widetilde{X}$ from $\Lambda$. So, this linear projection maps $\langle S(b) \rangle$ onto a plane $\mathbb{P}^2 = \mathbb{F} \subset \mathbb{P}^{a+3}$. From now on, we assume that this projection maps $S(b)$ birationally onto a plane curve $C_b \subset \mathbb{F}$ of degree $b$. Then as in (B), one can see that $X \subset \mathbb{P}^r$ is a surface of maximal sectional regularity of type II and $\mathbb{F}(X) = \mathbb{F}$. If $b \leq a+2$, then we have $\tau(X) = (2,3)$ by Theorem \ref{4.14'' Theorem}(c).

(D) From now on, we assume that $b \geq a+3$, and we will vary the projection center $\Lambda$. To do so, we first consider the canonical isomorphism
\begin{equation*}
\kappa: \mathbb{P}^1  \rightarrow  S(b), \quad [s:t] \mapsto [0:\ldots:0:s^b:s^{b-1}t:\ldots:st^{b-1}:t^b] \in \mathbb{P}^{a+b+1}.
\end{equation*}
Then, we choose a homogeneous polynomial $f \in \Bbbk[s,t]$ of degree $b$ which is not divisible by $s$ and by $t$. Now, let
$$\Lambda_f = \mathbb{P}^{b-3} \subset \langle S(b)\rangle \setminus S(b)$$
be such that the composition map
$$\varphi_f := \pi_{\Lambda_f} \circ \kappa: \mathbb{P}^1 \rightarrow C_b \subset \mathbb{F} = \mathbb{P}^2$$
of the linear projection map $$\pi_{\Lambda_f}: \mathbb{P}^{a+b+1}\setminus \Lambda_f \twoheadrightarrow \mathbb{P}^{a+3}$$ with the above map $\kappa$ sends $[s:t]$ to $[s^b : f : t^b]$.
Let
$$X_f := \pi_{\Lambda_f}(\widetilde{X}) \subset \mathbb{P}^{a+2}$$
denote the image of the scroll $\widetilde{X}$ under the linear projection map $\pi_{\Lambda_f}$ centered at $\Lambda_f$. Then, we may write
\begin{equation*}
X_f := \{ [us^a:us^{a-1}t:\ldots:ust^{a-1}:ut^a:vs^b:vf(s,t):vt^b ] \mid (s,t), (u,v) \in K^2 \setminus\{(0,0)\}\}.
\end{equation*}
After an appropriate choice of $f$, this latter presentation is accessible to syzygetic computations.

\end{construction and examples}
\vspace{0.2 cm}

\begin{example}\label{7.3 Example} Let $(a,b)=(3,5)$ and $f := s^4t+s^3t^2+s^2t^3+st^4$. Then $X_f \subset \mathbb{P}^6$ is of degree $d = 8 (= 2r-4)$ and the graded
Betti numbers $\beta_{i,j} = \beta_{i,j}(X)$ of $X$ are as presented in the following table.

\begin{table}[hbt]
\begin{center}
\begin{tabular}{|c|c|c|c|c|c|c|c|}\hline
                          \multicolumn{2}{|c||}{$i$}& $1$&$2$&$3$&$4$&$5$&$6$  \\\cline{1-8}
             \multicolumn{2}{|c||}{$\beta_{i,1}$}& $6$&$8$&$3$&$0$&$0$&$0$ \\\cline{1-1}
             \multicolumn{2}{|c||}{$\beta_{i,2}$}& $4$&$12$&$12$&$4$&$0$&$0$\\\cline{1-1}
             \multicolumn{2}{|c||}{$\beta_{i,3}$}& $0$&$0$&$0$&$0$&$0$&$0$\\\cline{1-1}
             \multicolumn{2}{|c||}{$\beta_{i,4}$}& $1$&$4$&$6$&$4$&$1$&$0$\\\cline{1-8}
             \end{tabular}
\end{center}
\end{table}

By Proposition~\ref{prop:BettiNumbers} (a) it follows from this graded Betti diagram of $X$, that
$$\tau(X)=(2,3).$$
\end{example}

\begin{example}\label{7.4 Example}
Let $(a,b)=(3,8)$ and consider $X_{f_i} \subset \mathbb{P}^6$ $(i=1,2,3)$ for the following choices of $f_i$:
\begin{enumerate}
\item[$(1)$] $f_1 = s^7t+s^6t^2+s^5t^3+s^4t^4+s^3t^5+s^2t^6+st^7$,
\item[$(2)$] $f_2 = s^7t+s^6t^2+s^5t^3+s^4t^4+s^3t^5+s^2t^6$, and
\item[$(3)$] $f_3 = s^7t+s^6t^2+s^5t^3+s^4t^4$.
\end{enumerate}
Then $X_{f_i} \subset \mathbb{P}^6$ is of degree $d = 11 \quad (= 2r-1 = 3r-7)$ for all $i=1,2,3$. The graded Betti diagrams of $X_{f_1}$, $X_{f_2}$ and
$X_{f_3}$ are given respectively in the three tables below.

\[
\begin{tabular}{|c|c|c|c|c|c|c|c|}\hline
                          \multicolumn{2}{|c||}{$i$}& $1$&$2$&$3$&$4$&$5$&$6$  \\\cline{1-8}
             \multicolumn{2}{|c||}{$\beta_{i,1}$}& $6$&$8$&$3$&$0$&$0$&$0$ \\\cline{1-1}
             \multicolumn{2}{|c||}{$\beta_{i,2}$}& $0$& $0$&$0$& $0$&$0$&$0$  \\\cline{1-1}
             \multicolumn{2}{|c||}{$\beta_{i,3}$}& $4$&$12$&$12$&$4$&$0$&$0$\\\cline{1-1}
             \multicolumn{2}{|c||}{$\beta_{i,4}$}& $0$&$0$&$0$&$0$&$0$&$0$\\\cline{1-1}
             \multicolumn{2}{|c||}{$\beta_{i,5}$}& $1$&$4$&$6$&$4$&$1$&$0$\\\cline{1-1}
             \multicolumn{2}{|c||}{$\beta_{i,6}$}& $0$&$0$&$0$&$0$&$0$&$0$\\\cline{1-1}
             \multicolumn{2}{|c||}{$\beta_{i,7}$}& $1$&$4$&$6$&$4$&$1$&$0$\\\cline{1-8}
             \end{tabular}
\]


\[
\begin{tabular}{|c|c|c|c|c|c|c|c|}\hline
                          \multicolumn{2}{|c||}{$i$}& $1$&$2$&$3$&$4$&$5$&$6$  \\\cline{1-8}
             \multicolumn{2}{|c||}{$\beta_{i,1}$}& $5$&$5$&$0$&$0$&$0$&$0$ \\\cline{1-1}
             \multicolumn{2}{|c||}{$\beta_{i,2}$}& $1$& $0$&$1$& $0$&$0$&$0$  \\\cline{1-1}
             \multicolumn{2}{|c||}{$\beta_{i,3}$}& $1$&$9$&$11$&$4$&$0$&$0$\\\cline{1-1}
             \multicolumn{2}{|c||}{$\beta_{i,4}$}& $4$&$18$&$32$&$28$&$12$&$2$\\\cline{1-1}
             \multicolumn{2}{|c||}{$\beta_{i,5}$}& $0$&$0$&$0$&$0$&$0$&$0$\\\cline{1-1}
             \multicolumn{2}{|c||}{$\beta_{i,6}$}& $0$&$0$&$0$&$0$&$0$&$0$\\\cline{1-1}
             \multicolumn{2}{|c||}{$\beta_{i,7}$}& $1$&$4$&$6$&$4$&$1$&$0$\\\cline{1-8}
             \end{tabular}
\]

\[
\begin{tabular}{|c|c|c|c|c|c|c|}\hline
                          \multicolumn{2}{|c||}{$i$}& $1$&$2$&$3$&$4$&$5$  \\\cline{1-7}
             \multicolumn{2}{|c||}{$\beta_{i,1}$}& $3$&$2$&$0$&$0$&$0$ \\\cline{1-1}
             \multicolumn{2}{|c||}{$\beta_{i,2}$}& $10$& $27$&$24$& $7$&$0$  \\\cline{1-1}
             \multicolumn{2}{|c||}{$\beta_{i,3}$}& $0$&$0$&$0$&$0$&$0$\\\cline{1-1}
             \multicolumn{2}{|c||}{$\beta_{i,4}$}& $0$&$0$&$0$&$0$&$0$\\\cline{1-1}
             \multicolumn{2}{|c||}{$\beta_{i,5}$}& $0$&$0$&$0$&$0$&$0$\\\cline{1-1}
             \multicolumn{2}{|c||}{$\beta_{i,6}$}& $0$&$0$&$0$&$0$&$0$\\\cline{1-1}
             \multicolumn{2}{|c||}{$\beta_{i,7}$}& $1$&$4$&$6$&$4$&$1$\\\cline{1-7}
             \end{tabular}
\]
By Proposition~\ref{prop:BettiNumbers} (a) we can see from these tables that
$$\tau (X_{f_1}) = (2,2), \quad \tau(X_{f_2}) = (1,1) \mbox{ and  } \tau (X_{f_3}) = (2,3).$$
\end{example}

\begin{example}\label{7.5 Example} Let $(a,b)=(3,9)$ and consider $X_{f_i} \subset \mathbb{P}^6$, $(i=1,2)$ for the two choices
\begin{enumerate}
\item[$(1)$] $f_1 = s^8t+s^7t^2+s^6t^3+s^5t^4+s^4t^5+s^3t^6+s^2t^7+st^8$ and
\item[$(2)$] $f_2 = s^8t+s^7t^2+s^6t^3+s^5t^4+s^4t^5+s^3t^6+s^2t^7$.
\end{enumerate}
Then $X_{f_i} \subset \mathbb{P}^6$ is of degree $d = 12 \quad (= 2r = 3r-6)$ for $i=1,2$. The graded Betti diagrams of $X_{f_1}$ and $X_{f_2}$ are
given respectively in the tables below.

\[
\begin{tabular}{|c|c|c|c|c|c|c|c|}\hline
                          \multicolumn{2}{|c||}{$i$}& $1$&$2$&$3$&$4$&$5$&$6$  \\\cline{1-8}
             \multicolumn{2}{|c||}{$\beta_{i,1}$}& $6$&$8$&$3$&$0$&$0$&$0$ \\\cline{1-1}
             \multicolumn{2}{|c||}{$\beta_{i,2}$}& $0$& $0$&$0$& $0$&$0$&$0$  \\\cline{1-1}
             \multicolumn{2}{|c||}{$\beta_{i,3}$}& $2$&$4$&$0$&$0$&$0$&$0$\\\cline{1-1}
             \multicolumn{2}{|c||}{$\beta_{i,4}$}& $1$&$4$&$10$&$6$&$1$&$0$\\\cline{1-1}
             \multicolumn{2}{|c||}{$\beta_{i,5}$}& $0$&$0$&$0$&$0$&$0$&$0$\\\cline{1-1}
             \multicolumn{2}{|c||}{$\beta_{i,6}$}& $1$&$4$&$6$&$4$&$1$&$0$\\\cline{1-1}
             \multicolumn{2}{|c||}{$\beta_{i,7}$}& $0$&$0$&$0$&$0$&$0$&$0$\\\cline{1-1}
             \multicolumn{2}{|c||}{$\beta_{i,8}$}& $1$&$4$&$6$&$4$&$1$&$0$\\\cline{1-8}
             \end{tabular}
\]
\[
\begin{tabular}{|c|c|c|c|c|c|c|c|}\hline
                          \multicolumn{2}{|c||}{$i$}& $1$&$2$&$3$&$4$&$5$&$6$  \\\cline{1-8}
             \multicolumn{2}{|c||}{$\beta_{i,1}$}& $5$&$5$&$0$&$0$&$0$&$0$ \\\cline{1-1}
             \multicolumn{2}{|c||}{$\beta_{i,2}$}& $0$& $0$&$1$& $0$&$0$&$0$  \\\cline{1-1}
             \multicolumn{2}{|c||}{$\beta_{i,3}$}& $5$&$15$&$15$&$5$&$0$&$0$\\\cline{1-1}
             \multicolumn{2}{|c||}{$\beta_{i,4}$}& $0$&$0$&$0$&$0$&$0$&$0$\\\cline{1-1}
             \multicolumn{2}{|c||}{$\beta_{i,5}$}& $5$&$23$&$42$&$38$&$17$&$3$\\\cline{1-1}
             \multicolumn{2}{|c||}{$\beta_{i,6}$}& $0$&$0$&$0$&$0$&$0$&$0$\\\cline{1-1}
             \multicolumn{2}{|c||}{$\beta_{i,7}$}& $0$&$0$&$0$&$0$&$0$&$0$\\\cline{1-1}
             \multicolumn{2}{|c||}{$\beta_{i,8}$}& $1$&$4$&$6$&$4$&$1$&$0$\\\cline{1-8}
             \end{tabular}
\]

\vspace{.2cm} 
By Proposition~\ref{prop:BettiNumbers} (a) we can verify that
$$\tau(X_{f_1}) = (2,2) \mbox{ and } \tau(X_{f_2}) = (1,1).$$
\end{example}

\begin{problem and remark}\label{7.6 Problem and Remark} (A) Let $5 \leq r <d$ and let $X \subset \mathbb{P}^r$ be a non-degenerate surface of degree $d$ which is of maximal
sectional regularity. We consider the three conditions
\begin{itemize}
\item[\rm{(i)}]   $N(X) \leq d-r$.
\item[\rm{(ii)}]  $\beta_{1,d-r+2}(X) = 1$.
\item[\rm{(iii)}] $\mathbb{F}(X) = \mathbb{P}^2$ or -- equivalently -- $X$ is of type II.
\end{itemize}
(B) By the implication (i) $\Rightarrow$ (iii) given in statement (a) of Theorem~\ref{4.17'' Proposition} we have the implication (i) $\Rightarrow$ (ii) among the
above three conditions. By the implication (i) $\Rightarrow$ (ii) given in statement (b) of Theorem~\ref{4.17'' Proposition} we have the implication (ii) $\Rightarrow$ (iii)
among the above three conditions. \\
We expect, that the converse of both implications holds but could not prove this. So we aim to pose the problem
\begin{itemize}
\item[\rm{(P)}] \textit{ Are the three conditions} (i), (ii) \textit{and} (iii) \textit{of part} (A) \textit{equivalent ?}
\end{itemize}
Observe, that in view of Remark~\ref{remark:normality} (A) an affirmative answer to this would also answer affirmatively the question, whether for surfaces of type II, the extended extremal variety
and the extremal variety of $X$ coincide (see Notation and Reminder~\ref{4.2'' Notation and Reminder}), hence the question whether
\begin{itemize}
\item[\rm{(Q)}] $\quad \mathbb{F}^{+}(X) = \mathbb{F}(X)$ \textit{for $X$ of type II ?}
\end{itemize}
Obviously, this latter question would been affirmatively answered if we could answer affirmatively the question
\begin{itemize}
\item[\rm{(R)}] $\quad {}^*\Sigma(X) = \Sigma(X)$ \textit{for $X$ of type II ?}
\end{itemize}
\end{problem and remark}

\subsection*{Acknowledgement} The first named author thanks to the Korea University Seoul, to the Mathematisches Forschungsinstitut Oberwolfach, to the Martin-Luther
Universit\"at Halle and to the Deutsche Forschungsgemeinschaft for
their hospitality and the financial support provided during the
preparation of this work.The second named author was supported by
Basic Science Research Program through the National Research
Foundation of Korea(NRF) funded by the Ministry of
Education(2014008404). The third named author was supported by the
NRF-DAAD GEnKO Program (NRF-2011-0021014). The fourth named author
thanks to the Korea University Seoul, to the Mathematisches
Forschungsinstitut Oberwolfach and to the Deutsche
Forschungsgemeinschaft for their hospitality respectively financial
support offered during the preparation of this work.

\end{document}